\documentclass[10pt]{amsart}
\usepackage{amsmath,amssymb,latexsym,soul,cite}
\usepackage{color,enumitem,graphicx}
\usepackage[colorlinks=true,urlcolor=blue,
citecolor=red,linkcolor=blue,linktocpage,pdfpagelabels,
bookmarksnumbered,bookmarksopen]{hyperref}
\usepackage[english]{babel}
\usepackage[T1]{fontenc}
\usepackage[hyperpageref]{backref}

\usepackage[left=2.65cm,right=2.65cm,top=3cm,bottom=3cm]{geometry}

\newcommand{\R}{{\mathbb R}}
\newcommand{\N}{{\mathbb N}}
\newcommand{\F}{{\mathcal F}}
\newcommand{\W}{{\mathcal W}}
\newcommand{\T}{{\mathcal T}}
\renewcommand{\O}{{\mathcal O}}
\renewcommand{\S}{{\mathcal S}}
\renewcommand{\phi}{{\varphi}}
\newcommand{\dist}{{\rm dist}}

\newcommand{\Tr}{{\rm Tr}}
\renewcommand{\div}{{\rm div}}
\newcommand{\eps}{{\varepsilon}}
\renewcommand{\epsilon}{{\varepsilon}}
\renewcommand{\theta}{{\vartheta}}

\numberwithin{equation}{section}

\newtheorem{theorem}{Theorem}[section]
\newtheorem{lemma}[theorem]{Lemma}
\newtheorem{corollary}[theorem]{Corollary}
\newtheorem{proposition}[theorem]{Proposition}

\newtheorem{remark}[theorem]{Remark}
\newtheorem{definition}[theorem]{Definition}

\definecolor{bluem}{rgb}{0.57, 0.64, 0.69}

\begin{document}

\title[Ground states for fractional 
H\'enon systems]{Asymptotics of ground states for fractional H\'enon systems}

\author[D.G.\ Costa]{David G. Costa}
\author[O.H.\ Miyagaki]{Ol\'{\i}mpio H. Miyagaki}
\author[M.\ Squassina]{Marco Squassina}
\author[J.\ Yang]{Jianfu\ Yang}

\address{Department of Mathematical Sciences  
\newline \indent  
University of Nevada Las Vegas, Box 454020, USA}
\email{costa@unlv.nevada.edu}

\address{Departamento de Matem\'atica 
\newline \indent  
Universidade Federal de Juiz de Fora
\newline\indent
Juiz de Fora, CEP 36036-330, Minas Gerais, Brazil}
\email{olimpio@ufv.br}

\address{Dipartimento di Informatica 
\newline \indent 
Universit\`a degli Studi di Verona
\newline\indent
Strada Le Grazie 15, I-37134 Verona, Italy}
\email{marco.squassina@univr.it}

\address{Department of Mathematics 
\newline \indent  
Jiangxi Normal University
\newline \indent
Nanchang, Jiangxi 330022, P.R. China}
\email{jfyang2000@yahoo.com}

\thanks{O.H.\ Miyagaki was partially supported by CNPq/Brazil and  
CAPES/Brazil (Proc 2531/14-3). J.\ Yang was supported by NNSF of China, No:11271170; GAN
 PO 555 program of Jiangxi and NNSF of Jiangxi, No:2012BAB201008.
 This paper was completed while the second author was visiting the
  Department of Mathematics  of  the Rutgers University, whose hospitality he gratefully
 acknowledges. He wishes to thank Professor H. Br\'ezis for invitation.}

%\date{\today}

%%    \thanks will become a 1st page (unnumbered) footnote.
%%    Multiple \thanks are possible. You may use it for the
%%    indication of the corresponding author.

%\date{\today}

\begin{abstract}
We investigate the asymptotic behavior of positive ground states for H\'enon type  systems
involving a fractional Laplacian on a bounded domain, when the powers of the nonlinearity approach the Sobolev
critical exponent. 
\end{abstract}

\dedicatory{Dedicated to Professor Djairo G.\ de Figueiredo on the occasion of his $80^{\mbox{{\tiny{th}}}}$ birthday}

\maketitle

%\VerbatimFootnotes
\section{Introduction and main results}

\noindent
Let $s\in (0,1)$, $N>2s$ and $B=\{x \in \R^N: |x|<1\}$.
Consider the fractional system of H\'enon type
\begin{equation}\label{eq:1.1}
\left\{
\begin{aligned}
(-\Delta)^s u &=\frac{2p}{p+q}|x|^\alpha u^{p-1}v^q && \text{in $B$,}\\
(-\Delta)^s v &=\frac{2q}{p+q}|x|^\alpha u^pv^{q-1} && \text{in $B$,}\\
 u>0&,\,\, v>0   && \text{in $B$,} \\
 u=&\, v=0 && \text{in $\partial B$,}
\end{aligned}
\right.
\end{equation}
where $(-\Delta)^{s}$ stands for the fractional Laplacian.
Recently, a great attention has been focused on the study of nonlinear problems involving the fractional  Laplacian,
 in view of concrete real-world applications. For instance, this type of operators arises in the thin obstacle problem, optimization, finance, phase
transitions, stratified materials, crystal dislocation, soft thin films, semipermeable membranes, flame
propagation, conservation laws, ultra-relativistic limits of quantum mechanics, quasi-geostrophic flows, multiple scattering,
minimal surfaces, materials science and water waves, see \cite{nezza}.
 In a smooth bounded domain $B \subset \R^N$, the operator $(-\Delta)^{s}$ can be defined by using the eigenvalues  $\{ \lambda_k\}$  and corresponding eigenfunctions $\{ \phi_k\}$ of the Laplace operator $-\Delta$ in $B$ with zero Dirichlet boundary values, normalized by $\|\phi_k\|_{L^2(B)}=1$, for all $k\in \N$, that is,
 $$
 - \Delta\phi_k=\lambda_k \phi_k \ \mbox{in}\ B, \qquad \phi_k=0 \ \mbox{on}\ \partial B.
 $$
 We define the space $H^{s}_{0}(B)$ by
 $$
 H^{s}_{0}(B):=\Big\{\displaystyle  u= \sum_{k=1}^{\infty} u_k
 \phi_k\, \text{ in $ L^{2}(B)$}:  \sum_{k=1}^{\infty}u^{2
}_{k} \lambda^{s}_{k}< \infty\Big\},
$$
equipped with the norm
$$
\|u\|_{H^{s}_{0}(B)}:=\Big(\sum_{k=1}^{\infty}u^{2}_{k} \lambda^{s}_{k}\Big)^{1/2}.
$$
Thus, for all $ u\in H^{s}_{0}(B)$, the fractional Laplacian $(-\Delta)^{s}$ can be defined as
$$
(-\Delta)^{s}u(x):= \sum_{k=1}^{\infty}  u_k \lambda^{s}_k\phi_{k}(x),\quad x\in B.
$$
We wish to point out that a different notion of fractional
Laplacian, available in the literature, is given by $(-\Delta)^{s}u=\F^{-1}(|\xi|^{2s}\F(u)(\xi)),$
 where $\F$ denotes the Fourier transform. This is also called the integral fractional Laplacian.
This definition, in bounded domains, is really different from the spectral one. 
In the case of the integral notion, due to the strong nonlocal character of the operator,  the Dirichlet
datum is given in $\R^N \setminus B$  and not simply on $\partial B.$
Recently, Caffarelli and Silvestre \cite{caffarelli} developed a local interpretation of the fractional
Laplacian given in $\R^N$ by considering a Dirichlet to Neumann type operator in the domain
$\{(x, t) \in \R^{N+1} : t > 0\}.$ A similar extension, in a bounded domain with
zero Dirichlet boundary condition,  was established, for instance,  by Cabr\`e and Tan in
\cite{cabretan}, Tan \cite{tan},  Capella, D\`avila, Dupaigne, and Sire \cite{capela}, and by Br$\ddot{\mbox{a}}$ndle,
Colorado, de Pablo, and S\`anchez \cite{colorado}. For any $ u \in H^{s}_{0}(B),$ the
solution $w \in H^{1}_{0,L}(C_{B})$ of 
 \begin{equation}
 \left\{ \begin{array}{rcl}
 -\div ( y^{1-2s}\nabla w)=0 & \mbox{in}& C_{B}:=B \times (0,\infty), \noindent\\
 w=0& \mbox{on}& \partial_L C_{B}:=\partial B \times (0,\infty),  \\
 w=u& \mbox{on}& B \times\{0\},\noindent
 \end{array}\right.
 \end{equation}
is called the $s$-harmonic  extension $w=E_{s}(u)$ of $u$, and it belongs to the space%where%
$$
H^{1}_{0,L}(C_{B})=\Big\{ w \in L^{2}(C_{B}): \text{$w = 0$ on  $\partial_{L} C_{B}$}:\,\,
\int_{ C_{B}}y^{1-2s}|\nabla w|^2 dxdy <\infty\Big\}.
$$
 %%and  $k_s=2^{1-2s}\Gamma(1-s)/\Gamma(s)$.
It is proved (see \cite[Section 4.1-4.2]{colorado}) that
$$
-k_{s}\lim_{y \to 0^+} y ^{1-2s}\frac{\partial w}{\partial y}=(-\Delta)^{s}u,
$$
where $k_s=2^{1-2s}\Gamma(1-s)/\Gamma(s)$.  Here $H^{1}_{0,L}(C_{B})$ is a Hilbert space endowed with the norm
$$
\|u\|_{H^{1}_{0,L}(C_{B})}= \Big(k_{s}
\int_{ C_{B}}y^{1-2s}|\nabla w|^2 dxdy\Big)^{1/2}.
$$
In the local case, the so-called H\'enon problem
\begin{equation*}
\left\{
\begin{aligned}
-\Delta u &=|x|^\alpha u^{p-1} && \text{in $B$,}\\
\,\,\, u>0& && \text{in $B$,} \\
\,\,\, u=0&  && \text{on $\partial B$,}
\end{aligned}
\right.\leqno{(HP)}
\end{equation*}
was first studied in \cite{Ni} after being introduced by
H\'enon in \cite{henon} in connection with the research of rotating stellar structures.
This problem has been studied by several authors, e.g. \cite{smets, cao, badialeserra} and references therein.
For  this class of problems, moving plane methods \cite{Gidas} cannot be applied,
and numerical calculations \cite{numerical} suggest that the existence of non-radial solutions is in fact possible.
In \cite{cao} the authors have shown
that the maximum point $x_p$ of a ground state solution for the H\'enon equation  $(HP)$  approaches a point $x_0\in \partial B$ as
 $p \to 2^*$, where $2^*=2N/(N-2)$. This result was extended to local H\'enon type variational systems  in \cite{jianfu}, as
 well as for scalar nonlocal H\'enon type equations in \cite{chenyang}.
The main goal of this paper to get a similar result for the nonlocal H\'enon system \eqref{eq:1.1}.
We reformulate the nonlocal systems \eqref{eq:1.1} into a local system, by
using the local reduction, that is, we set
$$
 \left\{ \begin{array}{rcl}
 -\div ( y^{1-2s}\nabla w_1)=0 & \mbox{in}& C_{B}=B \times (0,\infty), \noindent\\
-\div ( y^{1-2s}\nabla w_2)=0 & \mbox{in}& C_{B}=B \times (0,\infty), \noindent\\
 w_1=w_2=0& \mbox{on}& \partial_L C_{B}=\partial B \times (0,\infty),     \\
w_1=u\geq 0& \mbox{on}& B \times \{0\},     \\
w_2=v\geq 0& \mbox{on}& B \times \{0\},     \\
 k_s y ^{1-2s}\frac{\partial w_1}{\partial \nu}=\frac{2p}{p+q}|x|^\alpha u^{p-1}v^q& \mbox{on}& B \times\{0\}, \noindent\\
k_s y ^{1-2s}\frac{\partial w_2}{\partial \nu}=\frac{2q}{p+q}|x|^\alpha u^pv^{q-1}& \mbox{on}& B \times\{0\} . \noindent
 \end{array}\right.
\leqno{(LS)}
$$
Here $u(x)=w_1(x,0)$, $v(x)=w_2(x,0)$,      and the outward normal derivative should be understood as
$$
y ^{1-2s}\frac{\partial w}{\partial \nu}=-\lim_{y\to 0^{+}} y ^{1-2s}\frac{\partial w}{\partial y}.
$$
\vskip3pt
\noindent
Let us define the space $ H:= H^{1}_{0,L}(C_{B}) \times H^{1}_{0,L}(C_{B})$ and  the functional $I:H  \to \R$,
$$
I(w_1,w_2)=\frac{k_{s}}{2}\int_{C_{B}} y^{1-2s}(|\nabla w_1|^2  + |\nabla w_2|^2) dx dy-
 \frac{2}{p+q}\int_{B}|x|^{\alpha}w_1(x,0)^{p} w_2(x,0)^{q} dx.
 $$
 A weak solution to system $(LS)$
is a vector  $(w_1,w_2) \in  H$  verifying $I'(w_1,w_2)(h,k)=0$ for all $(h,k) \in  H$,
\begin{align*}
I^{\prime}(w_1,w_2)(h,k)&=k_s\int_{C_{B}} y^{1-2s}(\nabla w_1 \cdot\nabla h  + \nabla w_2 \cdot\nabla k) \,dx dy \\
&-\frac{2p}{p+q}\int_{B}|x|^{\alpha}w_1(x,0)^{p-1} w_2(x,0)^{q} h dx
-  \frac{2q}{p+q}\int_{B}|x|^{\alpha}w_1(x,0)^{p} w_2(x,0)^{q-1}k  dx.
\end{align*}
For the nonlocal scalar problem
\begin{equation}\label{eq:1.2}
\left\{
\begin{aligned}
(-\Delta)^s u &=|x|^\alpha u^{p} && \text{in $B$,}\\
u&>0 && \text{in $B$,} \\
u&=0 && \text{in $\partial B$,}
\end{aligned}
\right.
\end{equation}
we have
$$
 \left\{ \begin{array}{rcl}
 -\div ( y^{1-2s}\nabla w)=0 & \mbox{in}& C_{B}=B \times (0,\infty),  \noindent\\
 w=0& \mbox{on}& \partial_L C_{B}=\partial B \times (0,\infty), \\
k_{s} y ^{1-2s}\frac{\partial w}{\partial \nu}=|x|^\alpha u^{p-1}& \mbox{on}& B \times\{0\}.\noindent
 \end{array}\right.
\leqno{(LE)}
$$
For this problem consider the associated minimization problem
$$
\S^{\alpha}_{s,p}(C_{B})=
\inf_{ w \in H^{1}_{0,L}(C_{B})}\frac{\displaystyle k_s\int_{C_{B}}
 y^{1-2s}|\nabla w|^2 dx dy}{\Big(\displaystyle\int_{B} |x|^{\alpha}|w(x,0)|^p dx\Big)^{2/p}}.
 $$
Then $\S^{0}_{s,2^{*}_{s}}(C_{B})$, where $2^{*}_{s}:=2N/(N-2s)$, is never achieved \cite{colorado}
and $\S^{0}_{s,2^{*}_{s}}(\R^{N+1}_{+})$ is attained by the $w$ which are the $s$-harmonic extensions of
$$
u_{\epsilon}(x) = \frac{\epsilon^{\frac{N-2s}{4}}}{(\epsilon +|x|^2)^{\frac{N-2s}{2}}},\quad \eps>0,\,\, x\in\R^N.
$$
Let $U(x) =(1+|x|^2)^{\frac{2s-N}{2}}$ and let $W$ be the extension of $U$, namely
$$
\W(x,y)=E_{s}(U)= cy^{2s}\int_{\R^N} \frac{U(z)}{(|x-z|^2 + y^2)^{\frac{N+2s}{2}}} dz.
$$
For the system  $ (LS)$ consider the following minimization problem
\begin{equation}
\label{minproblemorig}
\S^{\alpha}_{s,p,q}(C_{B}) =\inf_{ w_1,w_2 \in H^{1}_{0,L}(C_{B})}\frac{\displaystyle k_s\int_{C_{B}}
 y^{1-2s}(|\nabla w_1|^2 +|\nabla w_2|^2)dx dy}{\Big(\displaystyle\int_{B} |x|^{\alpha}|w_1(x,0)|^p |w_2(x,0)|^q dx\Big)^{\frac{2}{p+q}}}
\end{equation}
\begin{theorem}\label{existence}
For any $\alpha>0$, $\S^{\alpha}_{s,p,q}(C_{B})$ is achieved if $2< p+q< 2^{*}_{s}$ .
\end{theorem}

\begin{proof}
Since $B$ is bounded and $\alpha>0$ we have $|x|^{\alpha} |u|^r \leq C|u|^r$.
The trace operator from $H^{1}_{0,L}(C_{B})$ to $L^r(B)$ is continuous
 if $1/r\geq 1/2- s/N,$ and compact if strict inequality holds, see \cite[Theorem 4.4]{colorado}
see also \cite{barrios,cabretan}.
Then the  trace operator $t_r:H^{1}_{0,L}(C_{B}) \to L^r(|x|^{\alpha},B)$ is compact for $r< 2N/(N- 2s).$
Taking a minimizing sequence $(w_{1,n},w_{2,n})$,  there is
$(w_1,w_2)\in H$ with $w_{i,n} \rightharpoonup w_i$,  as $n \to \infty.$ Then
\begin{align*}
& w_{1,n} \to w_1 \ \mbox{in}\ L^{p+q}(|x|^{\alpha}, B),\quad p+q < 2^{*}_{s},  \\
& w_{2,n} \to w_2 \ \mbox{in}\ L^{p+q}(|x|^{\alpha}, B),\quad p+q <2^{*}_{s}.
\end{align*}
By Young inequality we conclude that
$$
\int_{B} |x|^{\alpha}|w_{1,n}(x,0)|^p |w_{2,n}(x,0)|^q dx\to \int_{B} |x|^{\alpha}|w_1(x,0)|^p |w_2(x,0)|^q dx,
\ \mbox{as} \ n \to \infty.
$$
This implies that $\S^{\alpha}_{s,p,q}(C_{B})$ is achieved if $2< p+q< 2^{*}_{s}.$
\end{proof}

\noindent
\begin{remark}\rm
If $(w_{1,n},w_{2,n})$ is a minimizing sequence to $S^{\alpha}_{s,p,q}(C_{B})$, then it is readily seen that the sequence $(|w_{1,n}|,|w_{2,n}|)$ is
minimizing too. Thus, we can assume that the minimizer $(w_{1},w_{2})$ is non negative, that is,
 $ w_{1,n},w_{2,n}\geq 0.$ By maximum principle we have
$ w_{1,n},w_{2,n}> 0.$  Finally, invoking the
 regularity theory we infer that $ w_{1,n},w_{2,n}\in C^{\gamma}(C_{B})$, for some $\gamma \in (0,1).$
Notice that $(w_1, w_2)$ is a weak solution for $(LS)$.
Indeed, by Lagrange multiplier theorem, considering the constraint
$$
{\mathcal M}:=\Big\{(w_1,w_2) \in H: \int_{B}|x|^{\alpha}w_1(x,0)^{p} w_2(x,0)^{q} dx =1 \Big\},
$$ \noindent there exists $\lambda \in \R$  such
$$F'(w_1,w_2)(h,k)=\lambda G'(w_1,w_2)(h,k), \quad\forall (h,k) \in H,$$
where
\begin{align*}
F(w_1,w_2) &=\frac{k_s}{2}\int_{C_{B}} y^{1-2s}(|\nabla w_1|^2  + |\nabla w_2|^2) dx dy, \\
G(w_1,w_2)&=\int_{B}|x|^{\alpha}w_1(x,0)^{p} w_2(x,0)^{q} dx -1.
\end{align*}
Then, for all $(h,k) \in H,$ we have
\begin{eqnarray*}
\lefteqn{ k_s\int_{C_{B}} y^{1-2s}(\nabla w_1\cdot \nabla h  + \nabla w_2 \cdot\nabla k) dx dy}\\
&&= \lambda p\int_{B}|x|^{\alpha}w_1(x,0)^{p-1} w_2(x,0)^{q} hdx
+\lambda q\int_{B}|x|^{\alpha}w_1(x,0)^{p} w_2(x,0)^{q-1} kdx.
\end{eqnarray*}
By choosing $(h,k)=(w_1,w_2)$, we get
\begin{equation*}
 k_s\int_{C_{B}} y^{1-2s}(|\nabla w_1|^2  +|\nabla w_2|^2) dx dy =
\lambda (p+q)\int_{B}|x|^{\alpha}w_1(x,0)^{p} w_2(x,0)^{q} hdx
=\lambda (p+q).
\end{equation*}
Therefore $ \lambda >0$ and
$(\hat w_1,\hat w_2)=(\beta w_1,\beta w_2)$ with $\beta=\big(\frac{\lambda(p+q)}{2}\big)^{\frac{1}{p+q-2}}$ is a weak solution of $(LS).$
\end{remark}

\noindent
Now, we state the asymptotic behavior of ground states when $p+q \to 2^{*}_{s}.$

\begin{theorem}\label{main}
Let $\alpha >0$, $p_{\epsilon}$, $q >1$ with $ p_{\epsilon}+ q  < 2^{*}_{s}$,
$p_{\epsilon}\to p$ as $\epsilon \to 0$ and   $p+q=2^{*}_{s}.$
Let $(w_{1,\epsilon},w_{2,\epsilon})\in H$ be a solution to the minimization problem~\eqref{minproblemorig}. Then
there exists $x_0 \in \partial B$ such that
\begin{description}
 \item [i)] $k_sy^{1-2s}(|\nabla w_{1,\epsilon}|^2  +  |\nabla w_{2,\epsilon}|^2 ) \rightharpoonup \mu\delta_{(x_0,0)}$ in the sense of measure,
\item [ii)] $ |u_{1,\epsilon}|^{p_{\epsilon}}| u_{2,\epsilon}|^q  \rightharpoonup \gamma\delta_{x_0}$ in the sense of measure,
\end{description}
where $\mu>0, \gamma >0$ satisfy $ \mu \geq S \gamma^{2/2^{*}_{s}}$ and  $ \delta_{x_0}$ is the
Dirac mass at $x_0.$
\end{theorem}

\noindent
Let $(w_{1,\epsilon},w_{2,\epsilon})$ be a minimizer of  $\S^{\alpha}_{s,p_{\epsilon},q}(C_{B})$
which exists because $2< p_{\epsilon}+q <2^{*}_{s}.$ By regularity results (see e.g.\ \cite{colorado,cabre,capela}),
$(w_{1,\epsilon},w_{2,\epsilon})$ is H\"older continuous. We will show that there exists $ x_{\epsilon},\,y_{\epsilon} \in \overline{B}$ with
\begin{equation*}
M_{i,\epsilon} =w_{i,\epsilon}(x_{\epsilon},0)=\max_{(x,y)\in\overline{B}\times (0,\infty)}w_{i,\epsilon}(x,y).
\end{equation*}
Let $\lambda_{\epsilon} >0$ and $\bar\lambda_{\epsilon} >0$  be such that $\lambda_{\epsilon}^{\frac{N-2s}{2}}M_{1,\epsilon}=1$ and $\bar\lambda_{\epsilon}^{\frac{N-2s}{2}}M_{2,\epsilon}=1$,  where
$$
\lambda_{\epsilon}, \bar\lambda_{\epsilon} \to 0, \quad \mbox{as}\,\,  p_{\epsilon }+q \to 2^{*}_{s}.
$$
We  state  another description of the phenomenon exhibited in Theorem~\ref{main}.
\begin{theorem}\label{main2}
There hold
\begin{itemize}
\item [i)] $M_{1,\epsilon} = \O_\eps(1) M_{2,\epsilon}$ as $\epsilon\to 0$, hence, $\lambda_{\epsilon} = \O_\eps(1)\bar \lambda_{\epsilon}$ as $\epsilon\to 0$.
\item [ii)] $\displaystyle \dist(x_{\epsilon}, \partial B)\to 0 \ \mbox{and}\ \frac{ \dist(x_{\epsilon}, \partial B)}{\lambda_{\epsilon}}\ \to \infty\ \mbox{as}\  p_{\epsilon }+q \to 2^{*}_{s}$;
\item [iii)]$\displaystyle \lim_{ p_{\epsilon }+q \to 2^{*}_{s}}k_s\int_{C_{B}} y^{1-2s}(|\nabla \T_{1,\epsilon}|^2 +|\nabla \T_{2,\epsilon}|^2) dxdy=0$,
\end{itemize}
\end{theorem}
where we have set
$\T_{i,\epsilon}(x,y):= w_{1,\epsilon}(x,y)-\lambda_{\epsilon}^{\frac{2s-N}{2}} \W\big(\frac{x-x_{\epsilon}}{\lambda_{\epsilon}},
\frac{y}{\lambda_{\epsilon}}\big),$ for $i=1,2$.

\section{Preliminaries}
\label{prelim}
\noindent
For any $u_{i} \in H^{s}_{0}(B)$, there is a unique extension $w_{i}=E_s(u_{i}) \in
 H^{1}_{0,L}(C_{B})$ of $u_{i}$. The extension operator is an isometry between  $H^{s}_{0}(B)$ and $H^{1}_{0,L}(C_{B})$, that is
 (see \cite{barrios,colorado,chenyang})
$$
\|E_s(u_{i})\|_{ H^{1}_{0,L}(C_{B})}= \|u_{i}\|_{H^{s}_{0}(B)}, \quad i=1,2.
$$
Let us set
$$
x_0:=\Big( 1-\frac{1}{|\ln  \epsilon|},0, \ldots,0\Big) \in \R^N,
\,\,\quad z_0:=(x_0,0)\in \R^{N+1}.
$$
Let us denote $B_{\rho}:=\{x\in\R^N: |x-x_0|<\rho\}$ and
\begin{equation*}
 \mathbb{A}_{\rho}:=\{ (x,y)\in\R^{N+1}:  |(x,y)-z_0|< \rho\}, \quad
 \mathbb{B}_\rho:=\{(x,y)\in\R^{N+1}: |(x,y)| < \rho\}.
\end{equation*}
Let $\phi\in C^{\infty}_{0}(C_{B})$ be a cut-off function satisfying
$$
\phi(x,y):=\left\{ \begin{array}{ccr} 1&\mbox{if} &(x,y) \in \mathbb{A}_{\frac{1}{2|\ln  \epsilon|}}\\  0& \mbox{if} & (x,y) \not \in
\mathbb{A}_{\frac{1}{|\ln  \epsilon|}}, \end{array} \right.
$$
with $0\leq \phi(x,y)\leq 1$ and $ |\nabla \phi(x,y)|\leq C|\ln \epsilon|,$  for $(x,y) \in C_{B}.$
If $\W$ is the extension of the function $U$ previously introduced, we have (see \cite{barrios})
$|\nabla \W(x,y)|\leq Cy^{-1}\W(x,y)$,  for $(x,y) \in \R^{N+1}_{+}.$
The extension of $U_{\epsilon}(x) = (\epsilon +|x|^2)^{(2s-N)/2}$ has the form
$$
\W_{\epsilon}(x,y)=\epsilon^{\frac{2s-N}{2} } \W\Big(\frac{x-x_0}{\sqrt{\epsilon}}, \frac{y}{\sqrt{\epsilon}}\Big),\,\,\,\quad \epsilon >0.
$$
Notice that $\phi \W_{\epsilon} \in  H^{1}_{0,L}(C_{B})$ for $\eps$ small enough.
The following lemma is proved in \cite[Lemma 3.1]{chenyang}

\begin{lemma}
\label{stima1}
 There holds
$$
\frac{\displaystyle\int_{C_{B}}
k_s y^{1-2s}|\nabla(\phi  \W_{\epsilon})|^2 dx dy}{\Big(\displaystyle\int_{B} |x|^{\alpha}|\phi(x,0)\W_{\epsilon}(x,0)|^p dx\Big)^{2/p}}=\S^{0}_{s,2^{*}_{s}}(C_B)+
o_\epsilon(1),
$$
as $ p \to 2^{*}_{s},$ and $ \epsilon \to 0.$
\end{lemma}

\vskip2pt
\noindent
A minimizer of  $\S^{\alpha}_{s,p,q}(C_{B})$ exists as
 $2< p+q <2^{*}_{s}$
and arguing as in \cite[Theorem 5]{alves} we have
\begin{equation}\label{identidade}
\S^{\alpha}_{s,p,q}(C_{B})=
C_{p,q}\, \S^\alpha_{s,p+q}(C_{B}),
\,\,\,\quad
C_{p,q}:=\Big[\Big(\frac{p}{q}\Big)^{\frac{q}{p+q}}+
\Big(\frac{p}{q}\Big)^{-\frac{p}{p+q}}\Big],
\end{equation}
where we have set
$$
\S^{\alpha}_{s,p+q}(C_{B}) := \inf_{ w \in H^{1}_{0,L}(C_{B})}\frac{\displaystyle\int_{C_{B}}
k_s y^{1-2s}|\nabla w|^2 dx dy}{\Big(\displaystyle\int_{B} |x|^{\alpha}|w(x,0)|^{p+q} dx\Big)^{2/(p+q)}}.
$$
In particular
$$
\S_{s,p,q}(C_{B}):=\S^{0}_{s,p,q}(C_{B})=C_{p,q}\S^{0}_{s,p+q}(C_{B})=
C_{p,q}\S_{p+q}(C_{B}).
$$
Furthermore, if $w_0$ realizes $\S^{\alpha}_{s,p+q}(C_{B})$ then $(u_0,v_0)=(Bw_0,Cw_0)$ realizes
 $\S^{\alpha}_{s,p,q}(C_{B}),$
for
$$
B,C>0,\quad
B=\sqrt{p/q}\,C.
$$
Setting $\hat u_{\epsilon}=\sqrt{p_{\epsilon}}\phi \W_{\epsilon}$ and $\hat v_{\epsilon}=\sqrt{q}\phi \W_{\epsilon}$
and applying identity \eqref{identidade},  we have
\begin{align}
\label{3.16}
\frac{\displaystyle\int_{C_{B}} k_s y^{1-2s}(|\nabla \hat u_{\epsilon}|^2 +|\nabla
 \hat v_{\epsilon}|^2)dx dy}{\Big(\displaystyle\int_{B} |x|^{\alpha}|\hat u_{\epsilon}(x,0)|^{p_{\epsilon}}
 |\hat v_{\epsilon}(x,0)|^q dx\Big)^{2/(p_{\epsilon}+q)}}
 & =C_{p_{\epsilon},q}\frac{\displaystyle\int_{C_{B}}
k_s y^{1-2s}|\nabla(\phi \W_{\epsilon})|^2 dx dy}{\Big(\displaystyle\int_{B} |x|^{\alpha}|\phi(x,0)\W_{\epsilon}(x,0)|^{p_{\epsilon}+q} dx\Big)^{2/(p_{\epsilon}+q)}} \\
&= C_{p_{\epsilon},q}\S^{0}_{s,2^{*}_{s}}(C_B)+o_\epsilon(1),  \notag
\end{align}
as $ p_{\epsilon}+q \to 2^{*}_{s}$ for $ \epsilon \to 0.$
Following \cite[Lemma 3.2]{chenyang},  we have

\begin{lemma}\label{lemma3.2}
Let $(u_{\epsilon},v_{\epsilon})$ be a minimizer of
$\S^{\alpha}_{s,p_{\epsilon},q}(C_{B})$ and $p_{\epsilon} + q \to 2^{*}_{s}$ for $\epsilon \to 0$. Then we have
\begin{align*}
\lim_{\eps\to 0}\frac{\displaystyle\int_{C_{B}} k_s y^{1-2s}(|\nabla u_{\epsilon}|^2 + |\nabla v_{\epsilon}|^2)dx dy}{\Big(\displaystyle\int_{B} |x|^{\alpha}|u_{\epsilon}(x,0)|^{p_{\epsilon}} |v_{\epsilon}(x,0)|^q dx\Big)^{2/(p_{\epsilon}+q)}} = C_{p,q}\S^{0}_{s,2^{*}_{s}}(C_B)=C_{p,q}\S^{0}_{s,2^{*}_{s}}(\R^{N+1}_{+}),  \\
\lim_{\eps\to 0}\displaystyle \frac{\displaystyle\int_{C_{B}} k_s y^{1-2s}(|\nabla u_{\epsilon}|^2 + |\nabla v_{\epsilon}|^2)dx dy}{\Big(\displaystyle\int_{B}
 |u_{\epsilon}(x,0)|^{p_{\epsilon}} |v_{\epsilon}(x,0)|^q dx\Big)^{2/(p_{\epsilon}+q) }} = C_{p,q}\S^{0}_{s,2^{*}_{s}}(C_B)=C_{p,q}\S^{0}_{s,2^{*}_{s}}(\R^{N+1}_{+}).
\end{align*}
\end{lemma}

\begin{proof}
We already know that $\S^{0}_{s,2^{*}_{s}}(C_B)=\S^{0}_{s,2^{*}_{s}}(\R^{N+1}_{+})$. Notice that, by \eqref{identidade}, we get by Lemma~\ref{stima1}
\begin{align}
\notag
 \frac{\displaystyle\int_{C_{B}} k_s y^{1-2s}(|\nabla u_{\epsilon}|^2 + |\nabla v_{\epsilon}|^2)dx dy}{\Big(\displaystyle\int_{B}
 |u_{\epsilon}(x,0)|^{p_{\epsilon}} |v_{\epsilon}(x,0)|^q dx\Big)^{2/(p_{\epsilon}+q) }}  &\leq
\displaystyle \frac{\displaystyle\int_{C_{B}} k_s y^{1-2s}(|\nabla u_{\epsilon}|^2 + |\nabla v_{\epsilon}|^2)dx dy}{\Big(\displaystyle\int_{B}
|x|^{\alpha}|u_{\epsilon}(x,0)|^{p_{\epsilon}} |v_{\epsilon}(x,0)|^q dx\Big)^{2/(p_{\epsilon}+q)}} \\
&\leq  C_{p_{\epsilon},q}\frac{\displaystyle\int_{C_{B}}k_s
 y^{1-2s}|\nabla(\varphi  W_{\epsilon})|^2 dx dy}{\Big(\displaystyle\int_{B} |x|^{\alpha}|\varphi(x,0)W_{\epsilon}(x,0)|^{p_{\epsilon}+q} dx\Big)^{2/(p_\eps+q)}} \\
 &=  C_{p_{\epsilon},q} \S^{0}_{s,2^{*}_{s}}(C_B)+o(\epsilon),   \notag
\end{align}
as $ p_{\epsilon}+q \to 2^{*}_{s},$ for  $ \epsilon \to 0.$
On the other hand, we infer that
%\begin{equation*}
%\frac{\displaystyle\int_{C_{B}} (k_s y^{1-2s}(|\nabla u_{\epsilon}|^2 +
%|\nabla v_{\epsilon}|^2))dx dy}{\Big(\displaystyle\int_{B} |u_{\epsilon}(x,0)|^{p_{\epsilon}} |v_{\epsilon}(x,0)|^q dx\Big)^{2/(p_{\epsilon}+q) }
%}\geq \displaystyle \frac{\displaystyle\int_{C_{B}} (k_s y^{1-2s}(|\nabla u_{\epsilon}|^2 + |\nabla v_{\epsilon}|^2))dx dy}{\Big(\displaystyle\int_{B} |u_{\epsilon}(x,0)|^{p} |v_{\epsilon}(x,0)|^q dx\Big)^{2/2^{*}_{s} }}\\  \\
%=  C_{p,q}\S^{0}_{2^{*}_{s}}
%\end{equation*}
\begin{align*}
\frac{\displaystyle\int_{C_{B}} k_s y^{1-2s}(|\nabla u_{\epsilon}|^2 +
|\nabla v_{\epsilon}|^2)dx dy}{\Big(\displaystyle\int_{B} |u_{\epsilon}(x,0)|^{p_{\epsilon}} |v_{\epsilon}(x,0)|^q dx\Big)^{2/(p_{\epsilon}+q) }}
& \geq \S^{0}_{s,p_{\epsilon},q}(C_{B})=C_{p_{\epsilon},q}\S^{0}_{s,p_{\epsilon}+q}(C_{B}) \geq C_{p_{\epsilon},q}\S^{0}_{s,2^{*}_{s}}(C_B).
\end{align*}
The last inequality is due to H\"older inequality.
This concludes the proof.
\end{proof}

\begin{corollary}\label{noachieve}
Let $p+q=2^{*}_{s}$.\ Then the infimum $\S^{\alpha}_{s,p,q}(C_{B})$ cannot be achieved.
\end{corollary}

\begin{proof}
Observe that, for all $\alpha\geq 0$, there holds $\S^{\alpha}_{s,p,q}(C_{B})=C_{p,q}\, \S^\alpha_{s,2^*_s}(C_{B})$.
 Suppose, by contradiction, that $\S^{\alpha}_{s,p,q}(C_{B})$ is achieved by a function $(w_1,w_2)\in H.$ Without loss of generality, we may assume that
 $w_1\geq 0$ and $w_2\geq 0$.\ By Lemma \ref{lemma3.2}, we get
\begin{align*}
C_{p,q}\S^{0}_{s,2^{*}_{s}}(C_{B})&=\S^{\alpha}_{s,p,q}(C_{B})
=  \frac{\displaystyle \int_{C_{B}}
 k_s y^{1-2s}(|\nabla w_1|^2 + |\nabla w_2|^2)dx dy}{\displaystyle  \Big(\int_{B} |x|^{\alpha} w_1(x,0)^p  w_2(x,0)^q dx\Big)^{2/(p+q)}} \\
&\geq  \frac{\displaystyle \int_{C_{B}}
 k_s y^{1-2s}(|\nabla w_1|^2 + |\nabla w_2|^2)dx dy}{\Big(\displaystyle \int_{B} w_1(x,0)^p w_2(x,0)^q dx\Big)^{2/(p+q)}}
\geq  \S^{0}_{s,p,q}(C_{B})=C_{p,q}\S^{0}_{s,2^*_s}(C_{B}),
\end{align*}
so that $\S^{0}_{s,2^{*}_{s}}(C_{B})$ is achieved at $(w_1,w_2)\in H$, being
$$
C_{p,q}\S^{0}_{s,2^{*}_{s}}(C_{B})=\frac{\displaystyle \int_{C_{B}}
 k_s y^{1-2s}(|\nabla w_1|^2 + |\nabla w_2|^2) dx dy}{\Big(\displaystyle \int_{B} w_1(x,0)^p w_2(x,0)^q dx\Big)^{2/(p+q)}}.
 $$
By setting $\tilde w_i(x,t):=w_i(x,t)$ for $(x,t)\in B\times (0,\infty)$ and
$\tilde w_i(x,t):=0$ for $(x,t)\in \R^N\setminus B\times (0,\infty)$ we get the minimizer
$(\tilde w_1,\tilde w_2)\in\S^{0}_{2^{*}_{s}}(\R^{N+1}_{+})$. A contradiction, since $w_i>0$, by the maximum principle.
\end{proof}

\begin{definition} A sequence $(w_{1,n},w_{2,n})\subset H$ is said to be tight if,
for all $\eta>0$, there is $\rho_0 >0$ with
$$
\sup_{n\in\N}\int_{\{y>\rho_0\}}\int_{B} k_s y^{1-2s}(|\nabla w_{1,n}|^{2}+ |\nabla w_{2,n}|^{2}) dxdy \leq \eta.
$$
\end{definition}

\noindent
The following concentration compactness principle \cite{lions} can be adapted from \cite[Theorem 5.1]{barrios}

\begin{proposition}\label{prop}
Let $(w_{1,n},w_{2,n})\subset H$ be tight and weakly convergent to $(w_1,w_2)$ in $H$.
Let us denote $u_{i,n}=\Tr(w_{i,n})$  and $u_{i}=\Tr(w_{i})$,  $p+q=2^{*}_{s}.$
Let $\mu,\nu$ be two nonnegative measures such that
\begin{itemize}
 \item[i)] $k_s y^{1-2s}(|\nabla w_{1,n}|^2  +  |\nabla w_{2,n}|^2 ) \rightharpoonup \mu  $\,\,\, in the sense  of measure,
\item [ii)] $ |u_{1,n}|^{p}| u_{2,n}|^q  \rightharpoonup \nu$\,\,\, in the sense of measure.
\end{itemize}
Then  there exist an at most countable set I and points $\{x_i\}_{i\in I}\subset B$ such that
\begin{equation}
\label{measureb}
\mu\geq k_s y^{1-2s}(|\nabla w_1|^2 + |\nabla w_2|^2) + \sum_{k \in I}\mu_k\delta_{(x_k,0)},\quad \,\,
\nu= |u_{1}|^{p}| u_{2}|^q+ \sum_{k \in I}\nu_k\delta_{x_k},
\end{equation}
with $\mu_k>0, \nu_k>0$ and $\mu_k \geq  C_{p,q}\S^{0}_{s,2^{*}_{s}}\nu^{2/2^{*}_{s}}_{k}.$
\end{proposition}

\noindent
Finally, we give an explicit form to the solutions of the problem
\begin{equation}
\label{eq:1.1-p}
\left\{
\begin{aligned}
(-\Delta)^s u &=\frac{2p}{p+q} u^{p-1}v^q && \text{in $\mathbb{R}^N$},\\
(-\Delta)^s v &=\frac{2q}{p+q} u^pv^{q-1} && \text{in $\mathbb{R}^N$},\\
 u>0&,\,\, v>0   && \text{in $\mathbb{R}^N$}, \\
\end{aligned}
\right.
\end{equation}
where $p+q = 2^*_s$.
Let $u,v\in L^{2^*_s}(\mathbb{R}^N)$ be solutions of the following problem
\begin{equation}\label{eq:1.2-p}
\left\{
\begin{aligned}
u &=\frac{2p}{p+q}\int_{\mathbb{R}^N} \frac{u^{p-1}(y)v^q(y)}{|x-y|^{N-2s}}\,dy,\\
v &=\frac{2q}{p+q}\int_{\mathbb{R}^N}\frac{u^p(y)v^{q-1}(y)}{|x-y|^{N-2s}}\,dy.\\
 &u>0,\,\, v>0   \quad \text{in $\mathbb{R}^N$}. \\
\end{aligned}
\right.
\end{equation}
Denote by
$$
\tilde u(x):= \frac 1{|x|^{N-2s}}u\big(\frac x{|x|^2}\big),\qquad
\tilde v(x):= \frac 1{|x|^{N-2s}}v\big(\frac x{|x|^2}\big),
$$
the Kelvin transform of $u$ and $v$, respectively. Hence, $(\tilde u, \tilde v)$ is also a solution of \eqref{eq:1.2-p}.
We may prove as in \cite[Theorem 4.5]{CLO} that problems \eqref{eq:1.1-p} and \eqref{eq:1.2-p} are equivalent, 
that is if $(u,v)$ with $u,v \in H^s(\mathbb{R}^N)$ is a weak solution of \eqref{eq:1.1-p}, then $(u,v)$ is a  solution of \eqref{eq:1.2-p}, 
while if $(u,v)$ with $u,v \in L^{\frac {2N}{N-2s}}(\mathbb{R}^N)$ solves \eqref{eq:1.2-p}, 
then $(u,v)$ is a solution of \eqref{eq:1.1-p}.
%The following result is proved in \cite{ZL}.
%\begin{lemma}\label{lem:0.1}
% Let $(u,v)$ be a solution of \eqref{eq:1.2-p}. Then $u$ and $v$ are radially symmetric and monotone decreasing about the origin.
% {\color{red}{ Why necessarily about the origin? the problem is invariant under translations}}
%\end{lemma}
\noindent
Now we show that $L^{\frac {2N}{N-2s}}(\mathbb{R}^N)$ solution $(u,v)$ of the following problem is radially symmetric.
\begin{equation}\label{eq:1.2-p1}
\left\{
\begin{aligned}
u(x) &=\int_{\mathbb{R}^N} \frac{u^{p-1}(y)v^q(y)}{|x-y|^{N-2s}}\,dy,\\
v(x) &=\int_{\mathbb{R}^N}\frac{u^p(y)v^{q-1}(y)}{|x-y|^{N-2s}}\,dy.\\
 &u>0,\,\, v>0   \quad \text{in $\mathbb{R}^N$}. \\
\end{aligned}
\right.
\end{equation}

\noindent
Let $\Sigma_\lambda =\{x = (x_1,\cdots,x_N): x_1>\lambda\}$,\, $x^\lambda = (2\lambda-x_1,x_2,\cdots, x_N)$ and $u_\lambda(x) = u(x^\lambda)$.
\begin{lemma}\label{lem:0.2a}
Let $(u,v)$ be a solution of \eqref{eq:1.2-p1}. Then $(u,v)$ is radially symmetric with respect to some point.
\end{lemma}

\begin{proof} 
The result is proved by the moving plane methods developed for integral equations, see \cite{CL}. The argument is now standard, we sketch the proof. For details, we refer to similar arguments in \cite{ZL}. We have
\[
u_\lambda(x) - u(x) = \int_{\Sigma_\lambda}\Big(\frac 1{|x-y|^{N-2s}} - \frac 1{|x^\lambda-y|^{N-2s}}\Big)\Big(u_\lambda^{p-1}(y)v_\lambda^q(y) - u^{p-1}(y)v^q(y)\Big)\,dy
\]
and
\[
v_\lambda(x) - v(x) = \int_{\Sigma_\lambda}\Big(\frac 1{|x-y|^{N-2s}} - \frac 1{|x^\lambda-y|^{N-2s}}\Big)\Big(u_\lambda^{p}(y)v_\lambda^{q-1}(y) - u^{p}(y)v^{q-1}(y)\Big)\,dy.
\]
Next, we claim that there exist $K\geq 0$, such that if $\lambda< - K$, there holds
\[
u(x)\geq u_\lambda(x)\quad{\rm and }\quad v(x)\geq v_\lambda(x).
\]
Indeed, define
\[
\Sigma_\lambda^u =\{x\in \Sigma_\lambda: u(x)\leq u_\lambda(x)\},\quad \Sigma_\lambda^v =\{x\in \Sigma_\lambda: v(x)\leq v_\lambda(x)\}
\]
and $\Sigma_\lambda^- = \Sigma_\lambda\setminus \big(\Sigma_\lambda^u\cup \Sigma_\lambda^v\big)$, we can deduce as \cite{ZL} that
\[
u_\lambda(x) - u(x) \leq \int_{\Sigma_\lambda^v}\frac 1{|x-y|^{N-2s}} u_\lambda^{p-1}(y)\Big(v_\lambda^q(y)- v^q(y)\Big)\,dy + \int_{\Sigma_\lambda^u}\frac 1{|x-y|^{N-2s}} v_\lambda^q(y)\Big(u_\lambda^{p-1}(y) - u^{p-1}(y)\Big)\,dy.
\]
By the Hardy-Littlewood-Sobolev inequality,
\[
\|u_\lambda(x) - u(x)\|_{L^{2^*_s}(\Sigma_\lambda^u)}\leq C\|u_\lambda^{p-1}v_\lambda^{q-1}(v_\lambda-v)\|_{L^{\frac{2^*_sN}{N+2s2^*_s}}(\Sigma_\lambda^v)}
+ C\|u_\lambda^{p-2}v_\lambda^{q}(u_\lambda-u)\|_{L^{\frac{2^*_sN}{N+2s2^*_s}}(\Sigma_\lambda^u)}.
\]
By H\"{o}lder's inequality,
\[
\begin{split}
&\|u_\lambda(x) - u(x)\|_{L^{2^*_s}(\Sigma_\lambda^u)}\\
&\leq C\|u_\lambda\|^{p-1}_{L^{2^*_s}(\Sigma_\lambda^v)}\|v_\lambda\|^{q-1}_{L^{2^*_s}(\Sigma_\lambda^v)}
\|(v_\lambda-v)\|_{L^{2^*_s}(\Sigma_\lambda^v)}
+ C\|u_\lambda\|^{p-2}_{L^{2^*_s}(\Sigma_\lambda^u)}\|v_\lambda\|^{q}_{L^{2^*_s}(\Sigma_\lambda^u)}
\|(u_\lambda-u)\|_{L^{2^*_s}(\Sigma_\lambda^u)}.\\
\end{split}
\]
Choose $K>0$ large and for $\lambda<-K$, we have
\[
\|u_\lambda(x) - u(x)\|_{L^{2^*_s}(\Sigma_\lambda^u)}\leq \frac 14\|u_\lambda(x) - u(x)\|_{L^{2^*_s}(\Sigma_\lambda^u)} +\frac 14\|v_\lambda(x) - v(x)\|_{L^{2^*_s}(\Sigma_\lambda^v)}.
\]
Similarly,
\[
\|v_\lambda(x) - v(x)\|_{L^{2^*_s}(\Sigma_\lambda^u)}\leq \frac 14\|u_\lambda(x) - u(x)\|_{L^{2^*_s}(\Sigma_\lambda^u)} +\frac 14\|v_\lambda(x) - v(x)\|_{L^{2^*_s}(\Sigma_\lambda^v)}.
\]
The claim follows easily. Now, we may proceed as the proof of  \cite[Theorem 1.1]{ZL}.
\end{proof}

\noindent
It is known \cite{CLO} that a positive solution $U\in L^{2^*_s}(\mathbb{R}^N)$ of the problem
\begin{equation}\label{eq:1.3}
(-\Delta)^s u = u^{\frac{N+2s}{N-2s}} \quad\text{in $\mathbb{R}^N$,}
\end{equation}
is given by
\[
U(x) = C\Big(\frac t{t^2 + |x-x_0|^2}\Big)^{\frac{N-2s}2},
\]
for some constant $C = C(N,s)>0$, some $t>0$ and $x_0\in \mathbb{R}^N$.

\begin{lemma}\label{lem:0.2}
Let $(u,v)$ be a nontrivial weak solution of problem \eqref{eq:1.1-p}.
There exist $A, B>0$ such that $u = AU$ and $v= BU$.
\end{lemma}

\begin{proof}
We known that the solutions $(u,v)$ of \eqref{eq:1.1-p} are solutions of \eqref{eq:1.2-p}. 
By Lemma \ref{lem:0.2a}, any solution $(u,v)$ of \eqref{eq:1.2-p} is radially symmetric and monotone 
 decreasing about some point. Let $(\tilde u, \tilde v)$ be the Kelvin transform of $(u,v)$ with the pole $p\not=0$
 \[
 \tilde u(x) = \frac 1{|x-p|^{N-2s}} u\Big(\frac{x-p}{|x-p|^2}+p\Big),
 \,\,\quad 
 \tilde v(x) = \frac 1{|x-p|^{N-2s}} v\Big(\frac{x-p}{|x-p|^2}+p\Big).
 \]
We remark that $(\tilde u,\tilde v)$ is a solution of \eqref{eq:1.2-p} too, and then $(\tilde u,\tilde v)$ is radially symmetric with respect to some point $q$. Following the argument on page 280 in \cite{GL}, we can see that if $p=q$, then $(u,v)$ is constant, which is not true in our case. Hence, $p\not = q$. Now, using the Kelvin transform 
\[
K(f)(x) = \frac 1{|x|^{N-2s}} f\Big(\frac{x}{|x|^2}\Big),
\]
we deduce as in \cite[proof of  Lemma 7]{B} that $u = AU$ and $v= BU$. 
\end{proof}

\noindent
\section{Proof of  Theorem \ref{main}.}

\noindent
Choose $p_k$ such that $p_k +q \to 2^{*}_{s},$ as $k \to \infty$.
Let $(w_{1,k},w_{2,k})\in H$ be a nonnegative solution to
\begin{equation}
\label{min-probbloc}
\S^{\alpha}_{s,p_k,q}(C_{B})=\frac{\displaystyle k_s\int_{C_{B}}
 y^{1-2s}(|\nabla w_{1,k}|^2 +|\nabla w_{2,k}|^2)dx dy}{\Big(\displaystyle\int_{B} |x|^{\alpha}|w_{1,k}(x,0)|^{p_k} |w_{2,k}(x,0)|^q dx\Big)^{\frac{2}{p_k+q}}}.
\end{equation}
% of $\S^{\alpha}_{s,p_{\epsilon},q}(C_{B})$ as defined in \eqref{minproblemorig}.
Up to the factor $((p_k+q)\lambda_k/2)^{1/(p_k+q-2)}$ depending upon the Lagrange multiplier $\lambda_k$,
$(w_{1,k},w_{2,k})$ solves
\begin{equation}
\label{system-ww}
\left\{\begin{array}{l@{\quad }c}
 -\div(y^{1-2s}\nabla  w_{1,k}) = 0,\quad  -\div(y^{1-2s}\nabla  w_{2,k}) = 0, & \text{in $C_{B}$}, \\
k_s y ^{1-2s}\frac{\partial w_{1,k}}{\partial \nu}= \frac
{2p_k}{p_k+q} |x|^\alpha   w_{1,k}(x, 0) ^
{p_k-1}  w_{2,k}(x, 0)^q,  &\text{on $B\times \{0\}$}, \\
k_s y ^{1-2s}\frac{\partial  w_{2,k}}{\partial \nu}= \frac
{2q}{p_\epsilon+q}|x|^\alpha w_{1,k}(x, 0)^{p_k}  w_{2,k}(x, 0)^{q-1},& \text{on $B\times \{0\}$},\\
 w_{1,k} =  w_{2,k} = 0,& \text{on $\partial_L C_{B}$}.
\end{array}\right.
\end{equation}
In particular, we get
\begin{equation}
\label{ident}
\int_{C_{B}} k_s y^{1-2s}(|\nabla w_{1,k}|^2 + |\nabla w_{2,k}|^2)dx dy =2\int_B |x|^{\alpha} w_{1,k}(x,0)^{p_k} w_{2,k}(x,0)^q dx.
\end{equation}
%Now, let $(w_{1,k},w_{2,k})\in H$ be a minimizer for  the infimum $\S^{\alpha}_{s,p_{k},q}(C_{B})$. Whence
One may now set, for every $x\in B$ and $y>0$,
\begin{equation}
\label{connection}
\tilde w_{i,k}(x,y):=C_{k}w_{i,k}(x,y),  \,\,\quad
C_{k}=\Big(\int_{B}  w_{1,k}(x,0)^{p_k}w_{2,k}(x,0)^{q} dx\Big)^{-\frac{1}{p_k+q}},\quad i=1,2.
\end{equation}
%Thus, without loss of generality, we may assume that the $w_{i,k}$ satisfy
We have
$$
\int_{B}  \tilde w_{1,k}(x,0)^{p_k}\tilde w_{2,k}(x,0)^{q} dx=1,
\,\, \quad\text{for all $k\in\N$},
$$
and by \eqref{min-probbloc} and  Lemma \ref{lemma3.2}, we have
$$
\int_{C_{B}} k_s y^{1-2s}(|\nabla \tilde w_{1,k}|^2 + |\nabla \tilde w_{2,k}|^2) dx dy=  C_{p,q}\S^{0}_{s,2^{*}_{s}}+o_k(1),\quad \text{as $k\to\infty$.}
$$
The sequence
$C_{k}$ converges to some $C>0$, whenever $k\to\infty$.
This can be proved by comparison with the term $\int_B |x|^{\alpha} w_{1,k}(x,0)^{p_k} w_{2,k}(x,0)^q dx,$
which converges to a constant in view of formulas \eqref{min-probbloc}, \eqref{ident} and Lemma~\ref{lemma3.2}. In fact, taking into
account the Sobolev trace inequality, we have
$$
0<\sigma\leq \int_B |x|^{\alpha} w_{1,k}(x,0)^{p_k} w_{2,k}(x,0)^q dx\leq C_{k}^{-p_k-q}
\leq C\|w_{1,k}\|_{L^{2^*_s}(B)}^{p_k}\|w_{2,k}\|_{L^{2^*_s}(B)}^{q}\leq C.
$$
The sequence $(\tilde w_{1,k},\tilde w_{2,k})$ is bounded in  $H$. Furthermore, it is {\em tight}.
This fact can be proved by arguing as in \cite[Lemma 3.6]{barrios}.
By Proposition \ref{prop}, there exist nonnegative measures
 $\mu,\nu,$ a pair of  functions $ (w_1,w_2 ) \in H,$ an at most countable set $J$ and points with  $\{x_i\}_{i\in J}\subset B$ such that
\begin{itemize}
\item [i)] $\tilde w_{i,k} \rightharpoonup w_{i}$,\,\,\, $i=1,2,$
 \item[ii)] $k_s y^{1-2s}(|\nabla \tilde w_{1,k}|^2  + |\nabla \tilde w_{2,k}|^2 ) \rightharpoonup \mu $ \,\, in the sense  of measure,
\item [iii)] $ |\tilde w_{1,k}(x,0)|^{p}| \tilde w_{2,k}(x,0)|^q  \rightharpoonup \nu$\,\, in the sense of measure,
\end{itemize}
and \eqref{measureb} holds with $\nu_k>0$ and $\mu_k \geq C_{p,q} \S^{0}_{s,2^{*}_{s}}\nu^{2/2^{*}_{s}}_{k}.$ It follows that
%Taking $\phi \in L^{\infty}\cap C(\R^{N+1}_{+}),$  we have%
\begin{align*}
\lim_k \int_{\R^{N+1}_{+}}k_s y^{1-2s}(|\nabla \tilde w_{1,k}|^2  +  |\nabla \tilde w_{2,k}|^2 )\phi dxdy  & = \int_{\R^{N+1}_{+}}\phi \, d\mu,
\quad\forall \phi \in L^{\infty}\cap C(\R^{N+1}_{+}),\\
\lim_k \int_{\R^{N}} |\tilde w_{1,k}(x,0)|^{p}| \tilde w_{2,k}(x,0)|^q \phi dx & = \int_{\R^{N}}\phi \, d\nu,
\quad\forall \phi \in L^{\infty}\cap C(\R^N).
\end{align*}
In particular, we infer that
\begin{align*}
\lim_k \int_{\R^{N+1}_{+}} k_s y^{1-2s}(|\nabla \tilde w_{1,k}|^2  + |\nabla \tilde w_{2,k}|^2 ) dxdy & = \mu(\R^{N+1}_{+}), \\
\lim_k\int_{\R^{N}} |\tilde w_{1,k}(x,0)|^{p}| \tilde w_{2,k}(x,0)|^q  dx & = \nu(\R^N).
\end{align*}
\noindent {\em Claim:} $I\not=\emptyset$.

\noindent{\em Verification:} if $I=\emptyset$, we would have $\int_{B}|w_{1}(x,0)|^{p}|w_{2}(x,0)|^{q} dx=1$ and
\begin{align*}
C_{p,q}\S^{0}_{s,2^{*}_{s}}=&\lim_{k} \int_{C_{B}} k_s y^{1-2s}(|\nabla \tilde w_{1,k}|^2 +
 |\nabla \tilde w_{2,k}|^2) dx dy\\
 \geq &\int_{C_{B}}(k_s y^{1-2s}(|\nabla w_{1}|^2 +
|\nabla w_{2}|^2)) dx dy,
\end{align*}
yielding  $C_{p,q}\S^{0}_{s,2^{*}_{s}}=\int_{C_{B}} k_s y^{1-2s}(|\nabla w_{1}|^2 + |\nabla w_{2}|^2) dx dy,$
namely a contradiction to Corollary \ref{noachieve}.
\vskip3pt
\noindent {\em Claim:} $I$ contains only one point and   $w_1=w_2=0.$

\noindent{\em Verification:} We argue by contradiction and consider the following three cases:
\begin{itemize}
\item [i)] $w_1 \neq 0$ and $w_2 \neq 0$;
\item [ii)] $w_1 \neq 0$ and $w_2 = 0$;
\item [iii)] $w_1 =0$ and $w_2 \neq 0.$
\end{itemize}
In the case i), we have $\sum_{j\in J}\nu_j  \in (0,1).$ Notice that
\begin{align*}
C_{p,q}\S^{0}_{s,2^{*}_{s}}&=\lim_{k} \int_{C_{B}}k_s y^{1-2s}(|\nabla \tilde w_{1,k}|^2 +
 |\nabla \tilde w_{2,k}|^2) dx dy\\
 & \geq \int_{C_{B}}(k_s y^{1-2s}(|\nabla w_{1}|^2 + |\nabla w_{2}|^2)) dx dy +
C_{p,q} \S^{0}_{s,2^{*}_{s}}\sum_{j\in I}  \nu^{2/2^{*}_{s}}_{j},
\end{align*}
as well as
$$
1=\nu(\R^N)=\int_{B}|w_1(x,0)|^p |w_2(x,0)|^q dx +\sum_{j\in I}\nu_j .
$$
These facts  imply that
\begin{align*}
\int_{C_{B}} k_s y^{1-2s}(|\nabla w_{1}|^2 +|\nabla w_{2}|^2) dx dy &\leq  C_{p,q}\S^{0}_{s,2^{*}_{s}}\Big(1-
\sum_{j\in I} \nu^{2/2^{*}_{s}}_{j}  \Big) \\
& \leq C_{p,q}\S^{0}_{s,2^{*}_{s}}\Big(1-\sum_{j\in I} \nu_{j} \Big)^{2/2^{*}_{s}}\\
&=\S^{0}_{s,p,q}\Big(\int_{B}|w_1(x,0)|^p |w_2(x,0)|^q dx \Big)^{2/2^{*}_{s}},
\end{align*}
which is a contradiction. In the case ii) or iii), we have $\sum_{j\in J}\nu_j =1.$ Notice that
\begin{equation*}
C_{p,q}\S^{0}_{s,2^{*}_{s}}\geq \int_{C_{B}}(k_s y^{1-2s}(|\nabla w_{1}|^2 +
 |\nabla w_{2}|^2)) dx dy +C_{p,q} \S^{0}_{s,2^{*}_{s}}\sum_{j\in I} \nu^{2/2^{*}_{s}}_{j}.
\end{equation*}
This  implies, as above, that
\begin{equation*}
\int_{C_{B}} k_s y^{1-2s}(|\nabla w_{1}|^2 + |\nabla w_{2}|^2) dx dy \leq C_{p,q}\S^{0}_{s,2^{*}_{s}}\Big(1-\sum_{j\in I} \nu_{j}\Big)^{2/2^{*}_{s}}=0,
\end{equation*}
which is a contradiction. Then $w_1=w_2=0.$
We claim that $J$ is singleton. Notice again
\begin{align*}
1  \geq  \sum_{j\in I}  \nu^{2/2^{*}_{s}}_{j}
\geq  \Big(\sum_{j\in I}  \nu_{j}\Big)^{2/2^{*}_{s}}
=1,
\end{align*}
so there is at most one $j^{*} \in I$ such that $\nu_{j^{*}}\neq 0,$ proving the claim.
Hence there exists $x_0\in \overline{B}$ with
\begin{equation}
\label{concentration3}
  k_sy^{1-2s}(|\nabla \tilde w_{1,k}|^2  + |\nabla \tilde w_{2,k}|^2)  \rightharpoonup \mu_{0}\delta_{x_{0}}, \quad
 |\tilde w_{1,k}(x,0)|^{p}| \tilde w_{2,k}(x,0)|^q  \rightharpoonup \nu_{0}\delta_{x_{0}},
\end{equation}
in the sense of measure with $\mu_0 \geq C_{p,q}\S^{0}_{s,2^{*}_{s}}\nu^{2/2^{*}_{s}}_{0}.$ Taking into
account the relation \eqref{connection} between $\tilde w_{i,k}$ and $w_{i,k}$ the same conclusion
follows for the $w_{i,k}$.
\vskip3pt
\noindent
Assume by contradiction that $x_0\in B$. Then it follows $\dist(x_0, \partial B)=\sigma$, for $\sigma\in (0,1)$.
%there exists a constant $\Lambda(\sigma)\in (0,1)$ such that
%\begin{equation}
%\label{dill}
%\int_B |x|^{\alpha}|w_{1,k}(x,0)|^{p_{k}}|w_{2,k}(x,0)|^q dx\leq \Lambda(\sigma)^{2^*_s/2} \int_B |x|^{\alpha}|w_{1,k}(x,0)|^{p_{k}}|w_{2,k}(x,0)|^q dx+o_k^\eps(1)+o_\eps(1).
%\end{equation}
Notice that $|w_{1,k}(x,0)|^{p_k}\leq |w_{1,k}(x,0)|^{p}+o_k(1)$.
By the concentration behavior of the sequence $|w_{1,k}(x,0)|^{p}| w_{2,k}(x,0)|^q$ stated in  \eqref{concentration3}, there exists
$\varphi\in L^\infty\cap C(\R^N)$ with $\varphi(x_0)=0$ and
\begin{equation}
\int_{\R^N}  |w_{1,k}(x,0)|^{p}|w_{2,k}(x,0)|^q\chi_{B\setminus B(x_0,\sigma/2)}(x) dx\leq
\int_{\R^N}  |w_{1,k}(x,0)|^{p}|w_{2,k}(x,0)|^q \varphi(x)dx=o_k(1).
\end{equation}
Since $\int_{B}  |w_{2,k}(x,0)|^q dx\leq C$ by the Sobolev inequality \cite[formula (2.11)]{barrios}, then we conclude
\begin{align}
\label{dill}
 \int_B |x|^{\alpha}|w_{1,k}(x,0)|^{p_{k}}|w_{2,k}(x,0)|^q dx&=\int_{B(x_0,\sigma/2)} |x|^{\alpha}|w_{1,k}(x,0)|^{p_{k}}|w_{2,k}(x,0)|^q dx  \notag \\
&  +  \int_{\R^N}  |x|^{\alpha}|w_{1,k}(x,0)|^{p_{k}}|w_{2,k}(x,0)|^q \chi_{B\setminus B(x_0,\sigma/2)}(x) dx  \notag \\
   & \leq (1-\sigma/2)^\alpha\int_{B(x_0,\sigma/2)} |w_{1,k}(x,0)|^{p_{k}} |w_{2,k}(x,0)|^q dx  \\
  & +  \int_{\R^N}  |w_{1,k}(x,0)|^{p}|w_{2,k}(x,0)|^q \chi_{B\setminus B(x_0,\sigma/2)}(x)dx +o_k(1)  \notag \\
   & \leq \Lambda(\sigma)^{2^*_s/2} \Big(\int_{B} |w_{1,k}(x,0)|^{p_{k}} |w_{2,k}(x,0)|^q dx+o_k(1)\Big),  \notag
 \end{align}
 where $\Lambda(\sigma):=(1-\sigma/2)^{2\alpha/2^*_s}\in (0,1)$.
%which easily yields, by Lemma \ref{lemma3.2} the contradiction
%$C_{p,q} S^{0}_{2^{*}_{s}}> C_{p,q} S^{0}_{2^{*}_{s}}$.
%
By formula \eqref{dill}, on account of by Lemma \ref{lemma3.2}, it follows that
\begin{align*}
 C_{p,q}\S^{0}_{s,2^{*}_{s}} &=\lim_k\frac{\displaystyle\int_{C_B} k_sy^{1-2s}(|\nabla w_{1,k}|^2 + |\nabla w_{2,k}|^2)dxdy}{\Big(\displaystyle\int_B |x|^{\alpha}|w_{1,k}(x,0)|^{p_{k}}
 |w_{2,k}(x,0)|^q dx\Big)^{2/(p_{k}+q)}}  \\
&\geq  \frac{1}{\Lambda(\sigma)}\lim_k\frac{\displaystyle\int_{C_B} k_sy^{1-2s}(|\nabla w_{1,k}|^2 + |\nabla w_{2,k}|^2)dxdy}
 {\Big(\displaystyle\int_{B} |w_{1,k}(x,0)|^{p_{k}} |w_{2,k}(x,0)|^q dx\Big)^{2/(p_{k}+q)}}>C_{p,q}\S^{0}_{s,2^{*}_{s}} ,
\end{align*}
which is a contradiction, since $\Lambda^{-1}(\sigma)>1$. The proof of Theorem \ref{main} is complete. \qed

%\newpage
%{\color{blue}
\section{Proof of Theorem \ref{main2}}
\noindent
Let $(w_{1,\epsilon},w_{2,\epsilon})\in H$ be a nonnegative solution to~\eqref{min-probbloc}.
Then, up to a multiplicative constant depending upon the Lagrange multiplier, we may assume that
$(w_{1,\epsilon},w_{2,\epsilon})$ solves the system \eqref{system-ww}.
%\begin{equation*}
%\left\{\begin{array}{l@{\quad }c}
% -\div(y^{1-2s}\nabla  w_{1,\epsilon}) = 0,\quad  -\div(y^{1-2s}\nabla  w_{2,\epsilon}) = 0, & x\in
%C_{B}\\
%k_s y ^{1-2s}\frac{\partial w_{1,\epsilon}}{\partial \nu}= \frac
%{2p_\varepsilon}{p_\varepsilon+q} |x|^\alpha   w_{1,\epsilon}(x, 0) ^
%{p_\epsilon-1}  w_{2,\epsilon}(x, 0)^q,& x\in B,\\
%k_s y ^{1-2s}\frac{\partial  w_{2,\epsilon}}{\partial \nu}= \frac
%{2q}{p_\epsilon+q}|x|^\alpha w_{1,\epsilon}(x, 0)^{p_\epsilon}  w_{2,\epsilon}(x, 0)^{q-1},& x \in B,\\
% w_{1,\epsilon} =  w_{2,\epsilon} = 0,& x \in \partial_L C_{B},
%\end{array}\right.
%\end{equation*}
In particular, identity \eqref{ident} follows.
%\begin{equation}
%\label{ident}
%\int_{C_{B}} k_s y^{1-2s}(|\nabla w_{1,\epsilon}|^2 + |\nabla w_{2,\epsilon}|^2)dx dy =2\int_B |x|^{\alpha} w_{1,\eps}(x,0)^{p_\eps} w_{2,\eps}(x,0)^q dx.
%\end{equation}
Hence, from Lemma~\ref{lemma3.2}, we infer
\begin{equation}
\label{limit-norm}
\lim_{\eps\to 0} \displaystyle\int_{C_{B}} k_s y^{1-2s}(|\nabla w_{1,\epsilon}|^2 + |\nabla w_{2,\epsilon}|^2)dx dy
= \frac{(\S^{0}_{s,p,q}(C_B))^{\frac{N}{2s}}}{2^{\frac{N-2s}{2s}}}=
\frac{(\S^{0}_{s,p,q}(\R^{N+1}_+))^{\frac{N}{2s}}}{2^{\frac{N-2s}{2s}}}.
\end{equation}
We know that $(w_{1,\epsilon},w_{2,\epsilon})$ is a solution of the system
\begin{equation*}
\left\{\begin{array}{l@{\quad }c}
 -\div(y^{1-2s}\nabla  w_{1,\epsilon}) = 0,\quad  -\div(y^{1-2s}\nabla  w_{2,\epsilon}) = 0, & x\in
C_{B},\\
k_s y ^{1-2s}\frac{\partial w_{1,\epsilon}}{\partial \nu}= \frac
{2p_\varepsilon}{p_\varepsilon+q} |x|^\alpha   w_{1,\epsilon}(x, 0) ^
{p_\epsilon-1}  w_{2,\epsilon}(x, 0)^q,& x\in B,\\
k_s y ^{1-2s}\frac{\partial  w_{2,\epsilon}}{\partial \nu}= \frac
{2q}{p_\epsilon+q}|x|^\alpha w_{1,\epsilon}(x, 0)^{p_\epsilon}  w_{2,\epsilon}(x, 0)^{q-1},& x \in B,\\
 w_{1,\epsilon} =  w_{2,\epsilon} = 0,& x \in
\partial_L C_{B}.
\end{array}\right.
\end{equation*}
Then, we can assume $w_{i,\epsilon} \in C^{\tau}(B),$ for some $\tau \in (0,1)$.
There exist $ x_{1,\epsilon}, x_{2,\epsilon} \in \overline{B}$ such that
\begin{equation}
\label{Maxs}
M_{i,\epsilon} =w_{i,\epsilon}(x_{i,\epsilon},0)=\sup_{(x,y)\in\overline{B}\times (0,\infty)}w_{i,\epsilon}(x,y), \quad i=1,2.
\end{equation}
In fact, let $ x_{1,\epsilon}, x_{2,\epsilon} \in \overline{B}$ be such that
\begin{equation*}
M_{i,\epsilon}:=\sup_{x\in\overline{B} }w_{i,\epsilon}(x,0) =w_{i,\epsilon}(x_{i,\epsilon},0), \quad i=1,2.
\end{equation*}
Then the second equality in \eqref{Maxs} holds, since we have the following maximum principle

\begin{lemma}\label{lemma4.1}
$w_{i,\epsilon}(x,y)\leq M_{i,\epsilon}$ for a.e.\ $(x,y)\in B\times (0,\infty)$, for $i=1,2$.
\end{lemma}

\begin{proof}
Define $\tau_i(x,y) :=(w_{i,\epsilon}(x,y)- M_{i,\epsilon})^{+}$ for $(x,y)\in B\times (0,\infty).$ Then, testing in $(LS)$, we obtain
\begin{align*}
& k_s\int_{C_{B}}y^{1-2s}|\nabla \tau_{1}(x,y)|^2 dxdy=\frac{2p_{\epsilon}}{p_{\epsilon}+q}
\int_{B}|x|^{\alpha}w_{1,\epsilon}(x,0)^{p_{\epsilon}-1}w_{2,\epsilon}(x,0)^{q}\tau_{1}(x,0) dx=0, \\
& k_s\int_{C_{B}}y^{1-2s}|\nabla \tau_{2}(x,y)|^2 dxdy=\frac{2q}{p_{\epsilon}+q}\int_{B}|x|^{\alpha}w_{1,\epsilon}(x,0)^{p_{\epsilon}}w_{2,\epsilon}(x,0)^{q-1}\tau_{2}(x,0) dx=0.
\end{align*}
Then $\tau_{i}\equiv 0$ for $i=1,2$, yielding he conclusion.
\end{proof}

\begin{lemma}\label{lemma4.2} For all $i=1,2$,
we have $M_{i,\epsilon} \to +\infty$ as $p_{\epsilon}+q \to 2^{*}_{s}$.
\end{lemma}

\begin{proof}
Suppose by contradiction that there exist $C>0$ and a sequence $\{\epsilon_n\}\subset\R^+$ such that
$p_{\epsilon_n} +q \to 2^{*}_{s}$ and $M_{2,\epsilon_n}\leq C$, for all $n\in\N$.
Since $(w_{i,\epsilon_n})$ is bounded in $H^{1}_{0,L}(C_{B})$, up to a subsequence, by
the conclusions of Theorem \ref{main} we get $w_{i,\epsilon_n}\rightharpoonup 0$ weakly in $H^{1}_{0,L}(C_{B})$ and
$w_{i,\epsilon_n}\to 0$ in $L^r(B)$, for every $r<2^*_s$. Then, from identity \eqref{ident} and formula \eqref{limit-norm},
there exists a positive constant $\sigma$ independent of $\eps_n$ such that
$$
0<\sigma\leq\int_{B}|x|^{\alpha}|w_{1,\eps_n}(x,0)|^{p_{\epsilon}}|w_{2,\eps_n}(x,0)|^{q}dx
\leq \|w_{2,\eps_n}(x,0)\|^{q}_{\infty}\|w_{1,\eps_n}(x,0)\|_{L^{p_{\eps_n}}(B)}^{p_{\eps_n}}\leq C o_n(1),
$$
which yields a contradiction.
\end{proof}

\noindent
Now we want to recall some general Poho\v zaev type identity.
Consider the following system
$$
 \left\{ \begin{array}{rcl}
 -\div ( y^{1-2s}\nabla w_1)=0, & \mbox{in}& C_{B}=B \times (0,\infty), \noindent\\
-\div ( y^{1-2s}\nabla w_2)=0, & \mbox{in}& C_{B}=B \times (0,\infty), \noindent\\
 w_1=w_2=0,& \mbox{on}& \partial_L C_{B}=\partial B \times (0,\infty),\\
 k_{s} y ^{1-2s}\frac{\partial w_1}{\partial \nu}=C_1 w_1(x,0)^{p-1}w_2(x,0)^q, & \mbox{on}& B \times\{0\},\noindent\\
k_{s} y ^{1-2s}\frac{\partial w_2}{\partial \nu}=C_2 w_1(x,0)^p w_2(x,0)^{q-1}, & \mbox{on}& B \times\{0\},\noindent
 \end{array}\right.
\leqno{(LG)}
$$
where $\frac{\partial }{\partial \nu}$  denotes the outward normal derivative, and $\nu$ is exterior normal vetor to $\partial B.$
For the scalar case, the next result was obtained in \cite{colorado}, while for the system we refer to \cite{choi}.

\begin{theorem}\label{Pohozaev}
Let $p+q=2^*_s$. Then system $(LG)$ does not admit any nontrivial nonnegative solution.
\end{theorem}

\noindent
The following non existence result is crucial for our argument.
Consider the following  problem
$$
 \left\{ \begin{array}{rcl}
 -\div ( y^{1-2s}\nabla w_1)=0, & \mbox{in}& \R^{N+1}_{++},\\
-\div ( y^{1-2s}\nabla w_2)=0, & \mbox{in}& \R^{N+1}_{++},\\
 w_1=w_2=0,& \mbox{on}& \{ x_N =0, y>0 \},\\
 k_s y ^{1-2s}\frac{\partial w_1}{\partial \nu}=C_1 w_1(x,0)^{p-1}w_2(x,0)^q,& \mbox{on}& \{ x_N >0, y=0 \},\\
k_s y ^{1-2s}\frac{\partial w_2}{\partial \nu}=C_2 w_1(x,0)^p w_2(x,0)^{q-1},& \mbox{on}& \{ x_N >0, y=0 \}, \\
 \end{array}\right.
\leqno{(LS)}
$$
where $C_1,C_2 >0,$ $\frac{\partial }{\partial \nu}$  denotes the outward normal derivative, $p+q =2^{*}_{s}$ and
$$
\R^{N+1}_{++}:=\{(x_1,x_2,\ldots,x_{N-1},x_N,y)\in \R^{N+1}:  x_N >0, y>0 \}.
$$

\begin{proposition}\label{nonexistence}
Let $w_1,w_2\in H^{1}_{0,L}(\R^{N+1}_{++})$ be a bounded solution of $(LS).$ Then $(w_1,w_2)=(0,0).$
\end{proposition}
\begin{proof}
%\footnote{{\color{red}{{\bf REVISED UP TO HERE} $\uparrow$}}}
Since  $(x,y)\cdot \nu=0$ on  $\partial \R^{N+1}_{++},$ one cannot apply directly Poho\v zaev identities.
Whence, we use the Kelvin transformation as in \cite{choi,FW} to study a new system set in a ball.
Let $w_i\in H^{1}_{0,L}(\R^{N+1}_{++})$ be a solution to system $(LS)$. Then, the Kelvin transformation
of $w_i$ is defined by
$$
\widetilde{w_i}(z)=
|z|^{2s-N} w_i\big(\frac{z}{|z|^{2}}\big), \quad  z \in \R^{N+1}_{++}.
$$
and from \cite[Proposition 2.6]{FW}
%this, by change of variables, we infer that,
%for all $\phi \in  H^{1}_{0,L}(\R^{N+1}_{++})$, it holds (cf.\ \cite{choi,FW})
%\begin{align}
%\label{vchange}
%  \int_{\R^{N+1}_{++}} y^{1-2s}\nabla \widetilde{w_i}(z)\cdot \nabla \widetilde{\phi}(z)dx dy&=
% \int_{\R^{N+1}_{++}} y^{1-2s}\nabla w_i(z)\cdot \nabla \phi(z)dx dy\\
% \int_{\R^{N}_+}   \widetilde{w_1}(x,0)^{p-1} \widetilde{w_2}(x,0)^{q} \widetilde{\phi}(x,0)dx &=
%\int_{\R^{N}_+}   w_1(x,0)|^{p-1} w_2(x,0)^{q} \phi(x,0)dx.
%\end{align}
we infer that $\widetilde{w_i}$ is also a solution to $(LS)$.
By \cite[Corollary 2.1, Proposition 2.4]{Jin}, there exists $\gamma \in (0,1)$
with $\widetilde{w_i}(z)\leq C|z|^{\gamma}$, for $z\in B_{1}(0)$.
Then there exists $C>0$ such that
\begin{equation}\label{2.5}
|w_i(z)|\leq C(1 +|z|^2)^{-\frac{N-2s+\gamma}{2}}, \quad \mbox{for all $z \in \R^{N+1}_{++}$}.
\end{equation}
Arguing as in \cite{chenyang}, denote by $B_{\frac{1}{2}}(\frac{e_N}{2}) \subset \R^N$ the ball centered at $\frac{e_N}{2}$ with radius
$\frac{1}{2}.$ Define
$$
v_i(z):=|z|^{2s-N}w_i\big(-(e_N,0)+ \frac{z}{|z|^2}\big), \quad \mbox{for all}\ z \in \overline{C_{B_{\frac{1}{2}}(\frac{e_N}{2})}}\setminus\{0\}.
$$
By means of \eqref{2.5}, for a positive constant $C$  and for $|z|$ small enough, we have
%\footnote{I cannot justity this}
$$
v_i(z)\leq C|z|^{\gamma},\quad \mbox{for all}\ z\in  \R^{N+1}_{++}\setminus \{0\}.
$$
Therefore, we may extend $v_i$ by $0$  at $0.$
Then, as above, $(v_1,v_2)$ is a weak solution of the system
$$
 \left\{ \begin{array}{rcl}
 -\div ( y^{1-2s}\nabla v_1)=0, & \mbox{in}& C_{B_{\frac{1}{2}}(\frac{e_N}{2})},\\
-\div ( y^{1-2s}\nabla v_2)=0, & \mbox{in}& C_{B_{\frac{1}{2}}(\frac{e_N}{2})},\\
 v_1=v_2=0,& \mbox{on}& \partial_{L} C_{B_{\frac{1}{2}}(\frac{e_N}{2})}, \\
 k_s y ^{1-2s}\frac{\partial v_1}{\partial \nu}=C_1 v_1(x,0)^{p-1} v_2 (x,0)^q,&
\mbox{on}& B_{\frac{1}{2}}(\frac{e_N}{2})\times \{0\},\\
k_s y ^{1-2s}\frac{\partial v_2}{\partial \nu}=C_2 v_1(x,0)^p v_2(x,0)^{q-1},& \mbox{on}&
B_{\frac{1}{2}}(\frac{e_N}{2})\times \{0\}. \\
 \end{array}\right.
\leqno{(LSB)}
$$
%where $C_1,C_2 >0,$ $\frac{\partial }{\partial \nu}$  denotes the outward normal derivative, $p+q =2^{*}_{s},$
Now, applying Theorem \ref{Pohozaev} to system $(LSB)$ we infer that $v_i=0, i=1,2,$ that is,
 $w_i=0, i=1,2$.
\end{proof}

\vskip2pt
\noindent
We are now ready to complete the proof. By Lemma \ref{lemma4.2}, we may assume
$$
M_{1,\epsilon}=w_{1,\epsilon}(x_{1,\epsilon},0)=\!\!
\sup_{(x,y)\in\overline{B}\times (0,\infty)}w_{1,\epsilon}(x,y)\to+\infty,
$$
We may assume $M_{1,\epsilon}\geq M_{2,\epsilon}$. Let $\lambda_{\epsilon} >0$  be such that
$$
\lambda_{\epsilon}^{\frac{N-2s}{2}}M_{1,\epsilon}=1,\qquad
0\leq \lambda_{\epsilon}^{\frac{N-2s}{2}}M_{2,\epsilon}\leq 1,
$$
where $\lambda_{\epsilon} \to 0, \ \mbox{as}\  p_{\epsilon }+q \to 2^{*}_{s}.$
\noindent
Define the scaled functions
\begin{align*}
\tilde w_{1,\epsilon}(x, y) &:=
\lambda_\varepsilon^{\frac{N-2s}2}w_{1,\epsilon}(\lambda_\epsilon x
+ x_{1,\epsilon}, \lambda_\epsilon y),  \\
\tilde w_{2,\epsilon}(x, y) &:=
\lambda_\varepsilon^{\frac{N-2s}2}w_{2,\epsilon}(\lambda_\epsilon x
+ x_{1,\epsilon}, \lambda_\epsilon y),
\end{align*}
$B_\epsilon := \{ x\in \mathbb{R}^N:
\lambda_\epsilon x + x_{1,\epsilon}\in B_1(0)\}$ and
$C_{B_\epsilon}:=B_\epsilon \times (0,\infty).$
Then
$(\tilde w_{1,\epsilon}(x, y), \tilde w_{1,\epsilon}(x, y))$ satisfies
\begin{equation*}
\left\{\begin{array}{l@{\quad }c}
 -\div(y^{1-2s}\nabla \tilde w_{1,\epsilon}) = 0,\quad  -\div(y^{1-2s}\nabla \tilde w_{2,\epsilon}) = 0, & x\in
C_{B_\epsilon}\\
0< \tilde w_{1,\epsilon}\leq 1, \,\,\, 0<\tilde w_{2,\epsilon}\leq 1,
\,\,\, \tilde w_{1,\epsilon}(0,0) = 1, & x\in C_{B_\epsilon}\\
k_s y ^{1-2s}\frac{\partial\tilde w_{1,\epsilon}}{\partial \nu}= \frac
{2p_\varepsilon}{p_\varepsilon+q} |\lambda_\epsilon x +
x_{1,\epsilon}|^\alpha \lambda_\epsilon^{N- \frac
{N-2s}2(p_\epsilon+q)} \tilde w_{1,\epsilon}(x, 0) ^
{p_\epsilon-1} \tilde w_{2,\epsilon}(x, 0)^q,& x\in B_\epsilon,\\
k_s y ^{1-2s}\frac{\partial \tilde w_{2,\epsilon}}{\partial \nu}= \frac
{2q}{p_\epsilon+q}|\lambda_\epsilon x + x_{1,\epsilon}|^\alpha
\lambda_\epsilon^{N- \frac
{N-2s}2(p_\epsilon+q)}\tilde w_{1,\epsilon}(x, 0)^{p_\epsilon}  \tilde w_{2,\epsilon}(x, 0)^{q-1},& x \in B_\epsilon,\\
 \tilde w_{1,\epsilon} = 0,\quad \tilde w_{2,\epsilon} = 0,& x \in
\partial B_\epsilon\cap\partial_L C_{B_\epsilon}.\\
\end{array}\right.
\end{equation*}
Suppose $x_{1,\epsilon}\to x_0$ for some $x_0\in\bar B_1(0)$. We claim that $x_0\in
\partial B_1(0)$. By contradiction, assume that $x_0\in B_1(0)$ and let $d := \frac 12
\dist(x_0,\partial B_1(0))$. Denote
$\mathcal{B}(0,r)=\{z\in\mathbb{R}^{N+1}:|z|<r\}.$
For $\epsilon>0$ small,  both $\tilde w_{1,\epsilon}$ and $\tilde w_{2,\epsilon}$ are well defined in
$\mathcal{B}(0,d/{\lambda_\epsilon})\cap \mathbb{R}^{N+1}_{+},$ and
\[
\sup_{(x,y)\in \mathcal{B}(0,\frac d{\lambda_\epsilon})\cap \mathbb{R}^{N+1}_{+}}
\tilde w_{1,\epsilon}(x, y) = \tilde w_{1,\epsilon}(0,0) =1,\quad
\sup_{(x,y)\in \mathcal{B}(0,\frac d{\lambda_\epsilon})\cap \mathbb{R}^{N+1}_{+}}\tilde w_{2,\epsilon}(x,y)\in (0,1].
\]
Since $M_{1, \varepsilon}\to +\infty$, we have $0\leq
\lambda_\varepsilon\leq 1$, for $\varepsilon>0$ small. Let
$$
h(\epsilon):=\lambda_\epsilon^{N- \frac
{N-2s}2(p_\epsilon+q)} \,\quad {\rm and}\,\quad
h(0):=\lim\limits_{p_\epsilon +q\to 2^*_{s}}h(\epsilon).
$$
Hence, $0\leq h(\epsilon)\leq 1$. Three possibilities may occur, namely
\newline
(1) $h(0) = 0$,
\newline
(2) $h(0) = \beta\in (0,1)$,
\newline
(3) $h(0) = 1$.
\newline
We show that any of these cases yields
a contradiction. We observe that, for any $R>0$,
$B_{R}(0)\subset B_{d/{\lambda_\varepsilon}}(0)$ for $\epsilon>0$ small enough.
By Schauder estimates  \cite{capela,colorado,caffarelli,Jin} there are $C>0$ and  $0<\theta <1$ with
$$
\|\tilde w_{1,\epsilon}\|_{C^{0,\theta}(\mathcal{B}(0,2R)\cap \mathbb{R}^{N+1}_{+})}\leq C,\quad \|\tilde w_{2,\epsilon}\|_{C^{0,\theta}(\mathcal{B}(0,2R)\cap \mathbb{R}^{N+1}_{+})}\leq C
$$
for $\epsilon$ small enough.\ By Arzel\`a-Ascoli's Theorem, there exist subsequences $\{\tilde w_{i,\epsilon_k}\}$ such
that $\tilde w_{i,\epsilon_k} \to w_i$ as $k\to\infty$, for $i = 1,2,$ in $C^{0,\theta_0}_{{\rm loc}}$ for some $\theta_0\in (0,\theta)$. Then, we derive that $(w_1, w_2)$
satisfies
\begin{equation}\label{eq:2.5}\left\{\begin{array}{l@{\quad }l}
-{\rm div}(y^{1-2s}\nabla w_1) = 0,\quad -{\rm div}(y^{1-2s}\nabla w_2) = 0,& \ {\rm in}\ \mathbb{R}^{N+1}_{+},\\[1mm]
k_{s}y^{1-2s}\frac{ \partial w_1}{\partial\nu}= \frac{2p}{2^*_{s}}|x_0|^{\alpha}h(0)
w_1^{p-1}(x,0)w_2^{q}(x,0),& \ {\rm on}\ \mathbb{R}^{N}\times\{0\},\\[1mm]
k_{s}y^{1-2s}\frac{ \partial w_2}{\partial\nu}= \frac{2q}{2^*_{s}}|x_0|^{\alpha}h(0)
w_1^{p}(x,0)w_2^{q-1}(x,0),& \ {\rm on}\ \mathbb{R}^{N}\times\{0\},
\end{array}\right. \end{equation}
and  $w_1(0,0)=1,\, 0\leq w_2\leq 1$. Moreover, $w_i\in H^1_{0,L}(\mathbb{R}_+^{N+1})$.
If $x_0 = 0$ or $h(0) = 0$ or $w_2=0$, we have
 $w_1\equiv 0$, which is impossible since $w_1(0, 0) = 1$.
Suppose $x_0\not=0$, $w_2\not\equiv 0$ and $h(0) = \beta\in (0,1]$. Then
\begin{equation}\label{eq:2.5b}\left\{\begin{array}{l@{\quad }l}
-{\rm div}(y^{1-2s}\nabla w_1) = 0,\quad -{\rm div}(y^{1-2s}\nabla w_2) = 0,& \ {\rm in}\ \mathbb{R}^{N+1}_{+},\\[1mm]
k_{s}y^{1-2s}\frac{ \partial w_1}{\partial\nu}= \frac{2p}{2^*_{s}}\Lambda
w_1^{p-1}(x,0)w_2^{q}(x,0),& \ {\rm on}\ \mathbb{R}^{N}\times\{0\},\\[1mm]
k_{s}y^{1-2s}\frac{ \partial w_2}{\partial\nu}= \frac{2q}{2^*_{s}}\Lambda
w_1^{p}(x,0)w_2^{q-1}(x,0),& \ {\rm on}\ \mathbb{R}^{N}\times\{0\},
\end{array}\right. \end{equation}
where $\Lambda = |x_0|^\alpha \beta\in (0,1)$. Setting
$$
\bar w_1 := \Lambda^{\frac 1{2^*_{s}-2}} w_1,
\,\,\,\quad
\bar w_2 := \Lambda^{\frac 1{2^*_{s} -2}} w_2,
$$
we have $0< \bar w_1\leq \Lambda^{\frac 1{2^*_{s} -2}}$,
$0<\bar w_2\leq \Lambda^{\frac 1{2^*_{s} -2}},\,\, \bar w_1(0,0) = \Lambda^{\frac 1{2^*_{s}-2}}$ and  $(\bar w_1, \bar w_2)$ satisfies
\begin{equation}\label{eq:2.6}\left\{\begin{array}{l@{\quad }l}
-{\rm div}(y^{1-2s}\nabla \bar w_1) = 0,\quad -{\rm div}(y^{1-2s}\nabla \bar  w_2) = 0,& \ {\rm in}\ \mathbb{R}^{N+1}_{+},\\[1mm]
k_{s}y^{1-2s}\frac{ \partial \bar  w_1}{\partial\nu}= \frac{2p}{2^*_{s}}
\bar  w_1^{p-1}(x,0)\bar  w_2^{q}(x,0),& \ {\rm on}\ \mathbb{R}^{N}\times\{0\}, \\[1mm]
k_{s}y^{1-2s}\frac{ \partial\bar  w_2}{\partial\nu}= \frac{2q}{2^*_{s}}
\bar  w_1^{p}(x,0) \bar w_2^{q-1}(x,0),& \ {\rm on}\ \mathbb{R}^{N}\times\{0\},
\end{array}\right. \end{equation}
Define $\S:=\S^{0}_{s,p,q}(\R^{N+1}_{+})$ and observe that
\begin{align*}
 \int_{\R^{N+1}_{+}}k_s y^{1-2s}(|\nabla \bar w_1|^2 + |\nabla \bar w_2|^2)dx dy &=2\int_{\R^N} \bar w_{1}^{p}(x,0) \bar w_{2}^q(x,0)dx.
% \int_{C_B}k_s y^{1-2s}(|\nabla w_{1,\eps}|^2 + |\nabla w_{2,\eps}|^2)dx dy &=2\mu_\eps\int_{B} |x|^\alpha w_{1,\eps}^{p_\eps}(x,0) w_{2,\eps}^q(x,0)dx=2\mu_\eps.
\end{align*}
Then, by formula \eqref{limit-norm}, we have
\begin{align}
\label{chain}
 \S^{\frac{N}{2s}}& \leq 2^{\frac{N-2s}{2s}}  \int_{\R^{N+1}_{+}}k_s y^{1-2s}(|\nabla \bar w_1|^2 + |\nabla \bar w_2|^2)dx dy  \notag \\
&= \Lambda^{\frac{N-2s}{2s}} 2^{\frac{N-2s}{2s}} \int_{\R^{N+1}_{+}}k_s y^{1-2s}(|\nabla  w_1|^2 + |\nabla  w_2|^2)dx dy \notag \\
&\leq  \Lambda^{\frac{N-2s}{2s}} \liminf_{\eps\to 0}2^{\frac{N-2s}{2s}} \int_{C_{B_\eps}}k_s y^{1-2s}(|\nabla  \tilde w_{1,\eps}|^2 + |\nabla  \tilde w_{2,\eps}|^2)dx dy  \\
%& =  L^{\frac{N-2s}{2s}} \liminf_{\eps\to 0}2^{\frac{N-2s}{2s}} \int_{C_{B}} k_s y^{1-2s}(|\nabla  w_{1,\eps}|^2 + |\nabla  w_{2,\eps}|^2)dx dy  \\
& =  \Lambda^{\frac{N-2s}{2s}} \liminf_{\eps\to 0}   2^{\frac{N-2s}{2s}} \displaystyle\int_{C_{B}} k_s y^{1-2s}(|\nabla  w_{1,\eps}|^2 + |\nabla  w_{2,\eps}|^2)
dx dy \notag \\
&= \Lambda^{\frac{N-2s}{2s}}  \S^{\frac{N}{2s}}< \S^{\frac{N}{2s}}, \notag
\end{align}
a contradiction. Then $x_0\in \partial B_1(0).$ We can straighten $\partial B$ in a neighborhood of $x_0$ by a non-singular $C^1$ change
of coordinates. Let $x_N=\psi(x')$ be the equation of $\partial B$, where $x'=(x_1,x_2,\ldots,x_{N-1}),$ $\psi\in C^1.$
Define new coordinate system given by $z_i=x_i$ for $i=1,\ldots, N-1$, $z_N=x_N- \psi(x')$ and $z_{N+1}=y$.
Let $d_{\epsilon}=\dist(x_{\epsilon},\partial B).$ For $p_{\epsilon}+ q \to 2^{*}_{s} $
as $\epsilon \to 0,$
$\tilde w_{i,\epsilon}$  are well  defined in $B(0,\frac{\delta}{\lambda_{\epsilon}})\cap \R^{N+1}_{+}\cap \{z_N >
-\frac{d_{\epsilon}}{\lambda_{\epsilon}}\}$ for some $\delta>0$ small enough. Moreover
 $$
 \sup_{B(0,\frac{\delta}{\lambda_{\epsilon}})\cap \R^{N+1}_{+}\cap \{z_N >
-\frac{d_{\epsilon}}{\lambda_{\epsilon}}\}} \tilde w_{1,\epsilon} (x,y)=\tilde w_{1,\epsilon} (0,0)=1,
\quad
\sup_{B(0,\frac{\delta}{\lambda_{\epsilon}})\cap \R^{N+1}_{+}\cap \{z_N >
-\frac{d{\epsilon}}{\lambda_{\epsilon}}\}} \tilde w_{2,\epsilon} (x,y)\in (0,1].
$$
We now have the following

\noindent{\em Claim:} $d_{\epsilon}/\lambda_{\epsilon}\to +\infty\ \mbox{as}\ \epsilon \to 0.$

\noindent{\em Verification:} Suppose by contradiction that $d_{\epsilon}/\lambda_{\epsilon}$ remains bounded
and $d_{\epsilon}/\lambda_{\epsilon}\to s$ for some $s\geq 0.$
By the previous argument, since $|x_0|=1,$  we get
$\tilde w_{i,\epsilon} \to \tilde w_{i}$ in $C^{0,\gamma}_{{\rm loc}},$
$\tilde w_1(0,0)=1$  and
\begin{equation}\label{eq:4.7b}\left\{\begin{array}{l@{\quad }l}
{\rm div}(y^{1-2s}\nabla \tilde w_1) = 0,\quad {\rm div}(y^{1-2s}\nabla \tilde  w_2) = 0,& \ {\rm in}\ \{(z_1,\ldots,z_N,z_{N+1}):
\ z_N >-s, z_{N+1}>0\},\\[1mm]
k_{s}y^{1-2s}\frac{ \partial\tilde  w_1}{\partial\nu}= \Lambda\frac{2p}{2^*_{s}}
\tilde  w_1^{p-1}(x,0)\tilde  w_2^{q}(x,0),& \ {\rm on}\ \{(z_1,\ldots,z_N,z_{N+1}):
\ z_N >-s, z_{N+1}=0\},\\[1mm]
k_{s}y^{1-2s}\frac{ \partial\tilde  w_2}{\partial\nu}= \Lambda\frac{2p}{2^*_{s}}
\tilde  w_1^{p}(x,0)\tilde w_2^{q-1}(x,0),& \ {\rm on} \ \{(z_1,\ldots,z_N,z_{N+1}):
\ z_N >-s, z_{N+1}=0\},\\[1mm]
\tilde w_1(z)=\tilde w_2(z)=0,& \ {\rm on}\ \{(z_1,\ldots,z_N,z_{N+1}):
\ z_N =-s, z_{N+1}>0\},\\[1mm]
\tilde w_1(z),\tilde w_2(z)\in (0,1],& \ {\rm in}\ \{(z_1,\ldots,z_N,z_{N+1}):
\ z_N >-s, z_{N+1}>0\}.
\end{array}\right. 
\end{equation}
By a translation, $(\tilde w_1, \tilde w_2)$ verifies
\begin{equation}\label{eq:4.7c}\left\{\begin{array}{l@{\quad }cl}
{\rm div}(y^{1-2s}\nabla \tilde w_1) = 0,\quad {\rm div}(y^{1-2s}\nabla \tilde  w_2) = 0,& \ {\rm in} & \ \R^{N+1}_{++},\\[1mm]
k_{s}y^{1-2s}\frac{ \partial\tilde  w_1}{\partial\nu}= \Lambda\frac{2p}{2^*_{s}}
\tilde  w_1^{p-1}(x,0)\tilde  w_2^{q}(x,0),& \ {\rm on}&  \{(z_1,\ldots,z_N,z_{N+1}):
\ z_N >0, z_{N+1}=0\},\\[1mm]
k_{s}y^{1-2s}\frac{ \partial\tilde  w_2}{\partial\nu}= \Lambda\frac{2p}{2^*_{s}}
\tilde  w_1^{p}(x,0)\tilde w_2^{q-1}(x,0),& \ {\rm on} &  \{(z_1,\ldots,z_N,z_{N+1}):
\ z_N >0, z_{N+1}=0\},\\[1mm]
\tilde w_1(z)=\tilde w_2(z)=0,& \ {\rm on}&   \{(z_1,\ldots,z_N,z_{N+1}):
\ z_N =0, z_{N+1}>0\},\\[1mm]
\tilde w_2(z)\in (0,1],\tilde w_1 (0,\ldots,s,0)=1,& \ {\rm in}&  \R^{N+1}_{++}.
\end{array}\right. \end{equation}
Since $\tilde w_{i}\in H^{1}_{0,L}(\R^{N+1}_{++})$, by Proposition \ref{nonexistence},
$(\tilde w_1, \tilde w_2)=(0,0)$, which violates $\tilde w_1 (0,\ldots,s,0) = 1.$
Then the claim follows and $C_{B_{\epsilon}}$ converges to the entire $\R^{N+1}_{++}$ as $\epsilon \to 0.$

\noindent{\em Claim.} $\Lambda=|x_0|h(0)=h(0)=1.$
We can assume $(\tilde w_{1,\epsilon}, \tilde w_{2,\epsilon}) \to
(\tilde w_{1},\tilde w_{2}),$
$ \mbox{as}\ \epsilon \to 0,$ and $(\tilde w_{1},\tilde w_{2})$ satisfies
\begin{equation}\label{eq:4.7}\left\{\begin{array}{l@{\quad }cl}
{\rm div}(y^{1-2s}\nabla \tilde w_1) = 0,\quad {\rm div}(y^{1-2s}\nabla \tilde  w_2) = 0,& \ {\rm in} & \ \R^{N+1}_{++},\\[1mm]
k_{s}y^{1-2s}\frac{ \partial\tilde  w_1}{\partial\nu}= \Lambda\frac{2p}{2^*_{s}}
\tilde  w_1^{p-1}(x,0)\tilde  w_2^{q}(x,0),& \ {\rm on}&  \R^N \times \{0\},
\\[1mm]
k_{s}y^{1-2s}\frac{ \partial\tilde  w_2}{\partial\nu}= \Lambda\frac{2p}{2^*_{s}}
\tilde  w_1^{p}(x,0)\tilde w_2^{q-1}(x,0),& \ {\rm on} &  \R^N \times \{0\},\\[1mm]
\tilde w_1(z)=\tilde w_2(z)=0,& \ {\rm on}&   \{0\}\times (0,\infty),
\\[1mm]
\tilde w_i(z)\in (0,1],\tilde w_i (0,0)=1, & \ {\rm in}&  \R^{N+1}_{++}.\\[1mm]
\end{array}\right.
\end{equation}
If $\tilde w_2\equiv 0$ or $0\leq \Lambda <1,$ we reach the contradiction either as in \eqref{chain} or by Proposition \ref{nonexistence}. Hence, $\Lambda=1$ and $\tilde w_2\not\equiv 0$. This implies $M_{1,\varepsilon}^{-1} \tilde w_{2,\varepsilon}(\lambda_\varepsilon x +
x_\varepsilon) \to v(x)\not=0,$
and then $1\geq M_{1,\varepsilon}^{-1} M_{2,\varepsilon} \to\sigma>0$ as
$\varepsilon\to 0$, that is $M_{1,\varepsilon} = \O(1)M_{2,\varepsilon}$.  This is $(i)$ of Theorem 1.4. \newline
Let $y_\varepsilon\in B_1(0)$ be such that $w_{2,\varepsilon} (y_\varepsilon) = \max_{B_1(0)} w_{2,\varepsilon}(y).$ We
define
$
\tilde w_{2,\varepsilon}(x) = (\bar
\lambda_\varepsilon)^{(N-2s)/2}w_{2,\varepsilon}(\bar\lambda_\varepsilon
x +y_\varepsilon),
$
where $\bar
\lambda_\varepsilon^{(N-2s)/2}w_{2,\varepsilon}(y_\varepsilon) = 1$.
Suppose $y_\varepsilon\to y_0$. Again, using a blow up argument, we get
$y_0\in \partial B_1(0)$. Then,
%\footnote{What guarantees this? do we have a uniqueness result as for s=1?
%Y. Guo, J. Liu, Liouville type theorems for positive solutions of elliptic systems in $\R^N$, Comm.
%Partial Differential Equations 33 (2008), 263-284} 
in light of Lemma~\ref{lem:0.2}, we have 
$$
\tilde w_1(x,y)=a \W_1(x,y),\qquad \tilde w_2(x,y)=b \W_1(x,y)
$$ 
for some positive numbers $a,b$ such that $a/b=\sqrt{p/q}$.
%and, up to a constant multiplier, this implies that $x_0 = y_0$.
%By the fact that $\tilde w_1=a \W_{\epsilon},$ $\tilde w_2=b \W_{\epsilon}$, we see that
%$\tilde w_i$  attain their maximum 1 at $(x,y)=(0,0)$, we have that
%$\epsilon=1$ and  $x_0=0$
%By uniqueness, we can assume $\tilde w_{i,\epsilon} \to \tilde w_{i},$ as $ \epsilon \to 0.$
Let $\tilde v_{i,\epsilon}=\tilde w_{i,\epsilon}-\tilde w_{i}. $  Then $\tilde v_{i,\epsilon} \rightharpoonup 0$ weakly in
$H^{1}_{0,L}(C_{\omega})$ for any $\omega \subset \R^{N+1}_{+}$ and
\begin{equation}\label{4.11}\left\{\begin{array}{l@{\quad }l}
{\rm div}(y^{1-2s}\nabla \tilde v_{1,\epsilon}) = 0,\quad {\rm div}(y^{1-2s}\nabla \tilde v_{2,\epsilon}) = 0,& \ {\rm in}\ C_{B_{\epsilon}},\\[1mm]
\kappa_{2s}y^{1-2s}\frac{ \partial \tilde v_{1,\epsilon}}{\partial\nu}= \frac{2p_\epsilon}{p_{\epsilon}+q}Q_{\epsilon}
\tilde  w_{1,\epsilon}^{p_\epsilon-1}(x,0)\tilde  w_{2,\epsilon}^{q}(x,0)
-\frac{p(N-2s)}{N}w_{1}^{p-1}w_{2}^{q}& \ {\rm on}\ B_{\epsilon}\times\{0\}\\[1mm]
\kappa_{2s}y^{1-2s}\frac{ \partial\tilde v_{2,\epsilon}}{\partial\nu}= \frac{2 q}{p_{\epsilon}+q}Q_{\epsilon}
\tilde  w_{1,\epsilon}^{p_\epsilon}(x,0) \tilde w_{2,\epsilon}^{q-1}(x,0)
-\frac{q(N-2s)}{N}w_{1}^{p}w_{2}^{q-1}
& \ {\rm on}\ B_{\epsilon}\times\{0\}\\[1mm]
\tilde v_{i,\epsilon}=-{\tilde w_{i}} ,& \ {\rm on}\ \partial_{L}C_{B_{\epsilon}},\\[1mm]
\end{array}\right.
\end{equation}
where we have set
$Q_{\epsilon} :=|\lambda_\epsilon x +
x_{1,\epsilon}|^\alpha h(\eps).
$
Multiplying the first equation in \eqref{4.11} by $\tilde v_{1,\epsilon}$ and $\tilde v_{2,\epsilon}$, respectively, integrating by parts,
and applying Br\'ezis-Lieb Lemma, as $p_{\epsilon}+q \to 2^{*}_{s},$  we have
\begin{align*}
 \lefteqn{ k_s\int_{C_{B_{\epsilon}}}y^{1-2s}(|\nabla \tilde v_{1,\epsilon}|^2 + |\nabla \tilde v_{2,\epsilon}|^2) dx dy}\\
&= \int_{B_{\epsilon}}  \Big(\frac{2p_\epsilon}{p_{\epsilon}+q}Q_{\epsilon} \tilde  w_{1,\epsilon}^{p_\epsilon-1}(x,0)
\tilde  w_{2,\epsilon}^{q}(x,0)
-\frac{p(N-2s)}{N} \tilde w_{1}^{p-1}(x,0) \tilde w_{2}^{q}(x,0) \Big)\tilde v_{1,\epsilon}(x,0) dx\\
& - k_s \int_{\partial_{L}B_{\epsilon}}y^{1-2s}
\frac{\partial \tilde v_{1,\epsilon}}{\partial \nu} w_1 dS \\
&+\int_{B_{\epsilon}}\Big(\frac{2q}{p_{\epsilon}+q} Q_{\epsilon} \tilde  w_{1,\epsilon}^{p_\epsilon}(x,0)\tilde  w_{2,\epsilon}^{q-1}(x,0)
-\frac{q(N-2s)}{N} \tilde w_{1}^{p}(x,0) \tilde w_{2}^{q-1}(x,0)\Big)\tilde v_{2,\epsilon}(x,0) dx\\
& - k_s \int_{\partial_{L}B_{\epsilon}}y^{1-2s}
\frac{\partial \tilde v_{2,\epsilon}}{\partial \nu} w_2 dS \\
&=\frac{p(N-2s)}{N} \int_{B_{\epsilon}}  Q_{\epsilon} \Big(\tilde  w_{1,\epsilon}^{p_\epsilon-1}(x,0)\tilde  w_{2,\epsilon}^{q}(x,0)
- \tilde w_{1}^{p-1}(x,0) \tilde w_{2}^{q}(x,0)\Big)\tilde v_{1,\epsilon}(x,0) dx\\
& - k_s \int_{\partial_{L}B_{\epsilon}}y^{1-2s}
\frac{\partial \tilde v_{1,\epsilon}}{\partial \nu} \tilde  w_1 dS + \frac{p(N-2s)}{N} \int_{B_{\epsilon}}(Q_{\epsilon}-1) \tilde  w_{1}^{p-1}(x,0)\tilde  w_{2}^{q}(x,0) dx\\
&+\frac{q(N-2s)}{N} \int_{B_{\epsilon}} Q_{\epsilon} \Big(\tilde  w_{1,\epsilon}^{p_\epsilon}(x,0)\tilde  w_{2,\epsilon}^{q-1}(x,0) dx
- \tilde w_{1}^{p}(x,0)  \tilde w_{2}^{q-1}(x,0))\Big) \tilde v_{2,\epsilon}(x,0) dx\\
& - k_s \int_{\partial_{L}B_{\epsilon}}y^{1-2s}
\frac{\partial \tilde v_{2,\epsilon}}{\partial \nu} \tilde w_2 dS  + \frac{q(N-2s)}{N} \int_{B_{\epsilon}}(Q_{\epsilon}-1) \tilde  w_{1}^{p}(x,0)\tilde  w_{2}^{q-1}(x,0)dx+ o_\eps(1),
\end{align*}
since $Q_{\epsilon}\to 1$ as $\epsilon \to 0$. Recalling that $\W_{\epsilon}$ decay at infinity, we obtain
\begin{align*}
&   k_s\int_{C_{B_{\epsilon}}}y^{1-2s}(|\nabla \tilde v_{1,\epsilon}|^2 + |\nabla \tilde v_{2,\epsilon}|^2) dx dy\\
&=\frac{p(N-2s)}{N} \int_{B_{\epsilon}}Q_{\epsilon}\Big( (\tilde v_{1,\epsilon}+\tilde w_1)^{p_\epsilon-1}(x,0)(\tilde v_{2,\epsilon}+\tilde w_2)^{q}(x,0)
- \tilde w_{1}^{p-1}(x,0) \tilde w_{2}^{q}(x,0)\Big)\tilde v_{1,\epsilon}(x,0) dx\\
&+\frac{q(N-2s)}{N} \int_{B_{\epsilon}}Q_{\epsilon}\Big( (\tilde v_{1,\epsilon}+\tilde w_1)^{p_\epsilon}(x,0)(\tilde v_{2,\epsilon}+\tilde w_2)^{q-1}(x,0)
- \tilde w_{1}^{p}(x,0)  \tilde w_{2}^{q-1}(x,0)\Big)\tilde v_{2,\epsilon}(x,0) dx +o_\eps(1).
\end{align*}
Inserting $\tilde w_{i,\epsilon}=\tilde v_{i,\epsilon}=\tilde w_{i}, $ and using the following inequalities (cf.\ \cite{brezis,WY})
\begin{align*}
& ||a+b|^{p} -|a|^p -|b|^p -pab(|a|^{p-2}+|b|^{p-2}|)|\leq C\left\{\begin{array}{rcl} |a||b|^{p-1}& \mbox{if}& |a|\geq |b|,\\
               |a|^{p-1}|b|& \mbox{if}& |a|\leq |b|,\quad 1\leq p\leq 3,\\
              \end{array}\right. \\
& ||a+b|^{p} -|a|^p -|b|^p -pab(|a|^{p-2}+|b|^{p-2}|)|\leq C (|a|^{p-2}|b|^2 +|a|^2 |b|^{p-2}),\quad p\geq 3,
\end{align*}
 we infer that
\begin{align}\label{4.12}
 \lefteqn{ k_s\int_{C_{B_{\epsilon}}}y^{1-2s}(|\nabla \tilde v_{1,\epsilon}|^2 + |\nabla \tilde v_{2,\epsilon}|^2) dx dy}\nonumber\\
&=\frac{p(N-2s)}{N} \int_{B_{\epsilon}}Q_{\epsilon} \tilde  v_{1,\epsilon}^{p}(x,0)\tilde  v_{2,\epsilon}^{q}(x,0)dx
+ \frac{q(N-2s)}{N} \int_{B_{\epsilon}}Q_{\epsilon} \tilde  v_{1,\epsilon}^{p}(x,0)\tilde  v_{2,\epsilon}^{q}(x,0)dx + o_\eps(1) \\
&=2 \int_{B_{\epsilon}}Q_{\epsilon} \tilde  v_{1,\epsilon}^{p}(x,0)\tilde  v_{2,\epsilon}^{q}(x,0)dx +o_\eps(1).\nonumber
\end{align}
By definition of $\S^{\alpha}_{s,p_\epsilon,q}$ and recalling that $\S^{\alpha}_{s,p_\epsilon,q}=\S+ o_\eps(1),$ we
get
$$k_s\int_{C_{B_{\epsilon}}}y^{1-2s}(|\nabla \tilde v_{1,\epsilon}|^2 + |\nabla \tilde v_{2,\epsilon}|^2) dx dy\geq
\S\Big(\int_{B_{\epsilon}}Q_{\epsilon} \tilde  v_{1,\epsilon}^{p}(x,0)\tilde
 v_{2,\epsilon}^{q}(x,0)dx\Big)^{\frac{2}{p_{\epsilon}+q}} + o_\eps(1).
$$
Assume by contradiction that
$$
\lim_{\eps\to 0} k_s\int_{C_{B_{\epsilon}}}y^{1-2s}(|\nabla \tilde v_{1,\epsilon}|^2 + |\nabla \tilde v_{2,\epsilon}|^2) dx dy= \rho>0.
$$
Then, we have
\begin{align*}
k_s\int_{C_{B_{\epsilon}}}y^{1-2s}(|\nabla \tilde v_{1,\epsilon}|^2 + |\nabla \tilde v_{2,\epsilon}|^2) dx dy =
2 \int_{B_{\epsilon}}Q_{\epsilon} \tilde  v_{1,\epsilon}^{p}(x,0)\tilde  v_{2,\epsilon}^{q}(x,0)dx +o(1)
\geq \frac{\S^{\frac{N}{2s}}}{2^{\frac{N-2s}{2s}}} + o_\eps(1).
 \end{align*}
By using a Br\'ezis-Lieb type Lemma and arguments similar to the ones above, we get
\begin{align*}
I( \tilde w_{1,\epsilon}, \tilde w_{2,\epsilon})
&= \int_{C_{B_{\epsilon}}} \frac{k_s}{2}y^{1-2s}(|\nabla\tilde v_{1,\epsilon}|^2  + |\nabla \tilde v_{2,\epsilon}|^2) dx dy-
 \frac{2}{p_\epsilon+q}\int_{B_{\epsilon}} Q_{\epsilon}\tilde v_{1,\epsilon}(x,0)^{p} \tilde v_{2,\epsilon}(x,0)^{q} dx\\
& +\int_{\R^{N+1}_+} \frac{k_s}{2}y^{1-2s}(a^2 +b^2)|\nabla  \W_1|^2 dx dy-
 \frac{2}{p_\epsilon +q}\int_{\R^N} |a\W_1(x,0)|^{p_\epsilon} |b\W_1(x,0)|^{q} dx+ o_\eps(1)\\
&\geq \frac{2s}{N}\frac{\S^{\frac{N}{2s}}}{2^{\frac{N-2s}{2s}}} + o_\eps(1).
\end{align*}
On the other hand,
\begin{align*}
I(\tilde w_{1,\epsilon}, \tilde w_{2,\epsilon}) & :=\frac{k_{s}}{2}\int_{C_{B_{\epsilon}}}
y^{1-2s}(|\nabla \tilde w_{1,\epsilon}|^2  + |\nabla \tilde w_{2,\epsilon}|^2) dx dy-
 \frac{2}{p+q}\int_{B_{\epsilon}} Q_\eps \tilde w_{1,\epsilon}^{p}(x,0) \tilde w_{2,\epsilon}^{q}(x,0) dx \\
& =\frac{s}{N}\frac{\S^{\frac{N}{2s}}}{2^{\frac{N-2s}{2s}}} + o_\eps(1),
 \end{align*}
a contradiction. Hence $\rho=0$, proving also Theorem \ref{main2}(ii).

%}

\medskip

\end{document}